\documentclass[12pt]{article}
\usepackage{amsmath}
\usepackage{amssymb}
\usepackage{theorem}
\usepackage{enumerate} 
\usepackage{hyperref}

\numberwithin{equation}{section}

\newtheorem{thm}{Theorem}[section]
\newtheorem{lemma}[thm]{Lemma}

\newtheorem{prop}[thm]{Proposition}
\newtheorem{cor}[thm]{Corollary}
{\theorembodyfont{\rmfamily}

\newtheorem{rmk}[thm]{Remark}
}

\newcommand{\hT}{{\widehat T}}
\newcommand{\hP}{{\widehat P}}
\newcommand{\hB}{{\widehat B}}
\newcommand{\hA}{{\widehat A}}
\newcommand{\hE}{{\widehat E}}
\newcommand{\hH}{{\widehat H}}
\newcommand{\hR}{{\widehat R}}

\newcommand{\tA}{{\widetilde A}}
\newcommand{\tB}{{\widetilde B}}
\newcommand{\tH}{{\widetilde H}}
\newcommand{\tL}{{\widetilde L}}
\newcommand{\tM}{{\widetilde M}}
\newcommand{\tQ}{{\widetilde Q}}
\newcommand{\tT}{{\widetilde T}}

\newcommand{\qed}{\hfill \mbox{\raggedright \rule{.07in}{.1in}}}
 
\newenvironment{proof}{\vspace{1ex}\noindent{\bf
Proof}\hspace{0.5em}}{\hfill\qed\vspace{1ex}}
\newenvironment{pfof}[1]{\vspace{1ex}\noindent{\bf Proof of
#1}\hspace{0.5em}}{\hfill\qed\vspace{1ex}}

\def\R{\mathbb{R}}

\def\Z{\mathbb{Z}}
\def\C{\mathbb{C}}
\def\T{\mathbb{T}}
\def\cB{\mathcal{B}}
\def\cH{\mathcal{H}}

\newcommand{\eps}{{\epsilon}}
\newcommand{\suma}{{\textstyle{\sum_a}}}
\newcommand{\spec}{\operatorname{spec}}
\newcommand{\supp}{\operatorname{supp}}

\title{Local large deviations for periodic infinite horizon Lorentz gases}

\author{
Ian Melbourne \thanks{Mathematics Institute, University of Warwick, Coventry, CV4 7AL, UK}
\and
{Fran\c{c}oise P\`ene}
\thanks{Univ Brest, Universit\'e de Brest, LMBA,
UMR CNRS 6205, 
6 avenue Le Gorgeu, 29238 Brest cedex, France}
\and
Dalia Terhesiu
\thanks{Mathematisch Instituut,
University of Leiden, Niels Bohrweg 1, 2333 CA Leiden, Netherlands}
}

\date{29 August 2021; updated 21 May 2022}

\begin{document}

 \maketitle

\begin{abstract}
We prove local large deviations for the periodic infinite horizon Lorentz gas viewed as a $\Z^d$-cover ($d=1,2$) of a dispersing billiard. In addition to this specific example, we prove a general result for a class of nonuniformly hyperbolic dynamical systems and observables associated with central limit theorems with nonstandard normalisation.
\end{abstract}

\section{Introduction}
\label{sec:intro}

Local large deviations (LLD) for one dimensional i.i.d.\ random variables that do not satisfy the classical central limit theorem (with the standard normalisation) but are in the domain of a stable law were recently obtained by Caravenna \& Doney~\cite[Theorem~1.1]{CaravennaDoney} and refined by Berger~\cite[Theorem 2.3]{Berger19}. Such results have been extended to  multivariate i.i.d.\ random variables in the domain of the stable laws by Berger in~\cite{Berger19b}. Roughly speaking, an LLD measures the probability that the sum of the random variables assumes precise, but asymptotically large values. In the absence of second and even first moments, the proofs are considerably harder.

For dynamical systems, the first LLD results in the absence of the classical central limit theorem were obtained in~\cite{MTsub}; they are as optimal as~\cite[Theorem 2.3]{Berger19}. The main shift in that paper is an analytic proof which overcomes the restriction of having independence.
Although promising, the results in~\cite{MTsub} are limited to the Gibbs Markov maps.
The aim of this paper is to prove an LLD estimate for
\emph{infinite} horizon periodic Lorentz maps, which were shown to satisfy a central limit theorem
with nonstandard normalisation by Sz{\'a}sz \& Varj{\'u}~\cite{SV07}. A crucial new ingredient of the proofs of the present LLD results consists of a new operator renewal technique on the Young tower for the billiard map.

Periodic dispersing billiards and Lorentz gases were introduced into ergodic theory and studied by~\cite{Sinai70}. For a general reference, see~\cite{ChernovMarkarian}.
We recall that the classical central limit theorem was proved in the finite horizon case by~\cite{BunimovichSinaiChernov91} and local, moderate and large deviations were recently
obtained in Dolgopyat \& N\'andori~\cite{DP17}. In the same work~\cite{DP17} the authors designed a strategy to prove the local limit theorem and mixing properties for group extensions (such as $\Z^d$) of
probability preserving flows by free flight functions with finite second moments. For a similar strategy but weaker results we refer to~\cite{AT20}.
The strategy in~\cite{DP17} consists of the systematic use of local large and moderate deviations
for the underlying probability preserving Poincar\'e map.
Their result applies to the \emph{finite} horizon Lorentz flow.
In that case, both the free flight and the roof function are bounded. The LLD obtained in this paper (see Theorem~\ref{thm:lld} below) is close to optimal (see Remark~\ref{rmk:opt} below)
and we believe that it can be used to prove the local limit theorem and mixing properties for
the \emph{infinite} horizon Lorentz flow.

An infinite horizon periodic Lorentz map $(\tT,\tM,\tilde\mu)$
is a $\Z^d$-cover of 
an infinite horizon periodic dispersing billiard $(T,M,\mu)$ where $\tilde\mu$ and $\mu$ are the Liouville measures on $\tM$ and $M$. 
The notation for the infinite horizon dispersing billiard
  is recalled in Section~\ref{sec:setup}. 
We consider the cases $d=1$ (tubular billiard) and $d=2$ (planar billiard). 

Let $\kappa:M\to\Z^d$ denote 
the free flight function between collisions,
and define $\kappa_n=\sum_{j=0}^{n-1}\kappa\circ T^j$.
For the Lorentz gas, geometrically $\kappa_n\in\Z^d$ denotes the cell in the infinite measure phase space $\tM$ where the $n$'th collision takes place for initial conditions starting in the $0$'th cell.
Set
\[
a_n=\sqrt{n\log n}.
\]
The central limit theorem with nonstandard normalisation proved in~\cite{SV07} says that $a_n^{-1}\kappa_n$ converges in distribution to a $d$-dimensional normal distribution. In fact,~\cite{SV07} proves a stronger result, namely the corresponding local limit theorem.
Our main result is:\footnote{We set $\log x=1$ for $x\in(0,2)$.}

\begin{thm}[LLD for the dispersing billiard]
\label{thm:lld} 
There exists $C>0$ such that
for all  $n\ge1$, $N\in\Z^d$,
\[
\mu(\kappa_n=N)\le \begin{cases}
 C \dfrac{n}{a_n} \, \dfrac{\log |N|}{1+|N|^2}, & d=1, \\[2.75ex]
C \dfrac{n}{a_n^2} \, \dfrac{\log |N|\log\log|N|}{1+|N|^2}, & d=2.
\end{cases}
\]
% Moreover, when $d=2$, there exists $\eps_1>0$, $C>0$ such that
% \[
% \mu(\kappa_n=N)\le C
% \frac{n}{a_n^2} \, \frac{\log |N|}{1+|N|^2}
% \quad\text{for all  $n\ge1$, $N\in\Z^2$ with $|N|\le e^{\eps_1 n}$.}
% \]
\end{thm}

\noindent Again, there is the geometric interpretation that $\mu(\kappa_n=N)$ represents the probability 
that an initial condition in the $0$'th cell of $\tM$ lies in the $N$'th cell after $n$ collisions.

\begin{rmk}
\label{rmk:opt}
 (a) 
A slightly stronger version of Theorem~\ref{thm:lld} for unrestricted $d$ with bound
$\frac{n}{a_n^d} \, \frac{\log |N|}{1+|N|^2}$ 
for all $n$ and $N$
 was proved in the much easier set up of Gibbs-Markov maps (and also for i.i.d.\ random variables) in~\cite[Corollary 3.3]{MTsub}.
 \\[.75ex]
(b) It follows from the arguments in this paper that the extra factor of $\log\log |N|$ for $d=2$ in Theorem~\ref{thm:lld} can be removed in the range $|N|\le e^{\eps_1 n}$, for $\eps_1$ sufficiently small as chosen in Section~\ref{sec:pfkey}. The same holds true in the range $n\ll \log N/\log\log N$ (see the end of the proof of Lemma~\ref{lem:bigN}).
 \\[.75ex]
(c)
For the dispersing billiard, the improved bound in~(a)
 follows from a uniform version~\cite{PT20} of the local limit theorem~\cite{SV07} in the range
 $N\ll \sqrt{n\log n}$.
 Hence, the principal novelty of Theorem~\ref{thm:lld} lies in the range $N\gg\sqrt{n\log n}$.  
 We note that, as in~\cite{MTsub}, the approach in this paper does not rely on the local limit theorem and extends to situations where the local limit theorem fails, see Theorem~\ref{thm:abs}.
\end{rmk}

The approach in this paper, following~\cite{MTsub}, is Fourier analytic and relies on smoothness properties of the leading eigenvalues and their spectral projections for the appropriate transfer operator. 
We show how to obtain $C^r$ control for all $r<2$,
 going considerably beyond previous estimates of~\cite{BalintGouezel06,PT20}. 
The methods developed in Section~\ref{sec:pfkey} to obtain this control 
in the context of exponential Young towers
are the main technical advance of this paper and should have other applications, not only to LLD.

In Section~\ref{sec:setup}, we recall the setting for dispersing billiards.
In Section~\ref{sec:bigN}, we prove Theorem~\ref{thm:lld}
in the range $n\ll\log|N|$.
Sections~\ref{sec:key} to~\ref{sec:lldpf} 
treat the complementary range $\log|N|\le\eps_1 n$ where $\eps_1$ is chosen sufficiently small.  
Key technical estimates are stated in Section~\ref{sec:key} and proved in Section~\ref{sec:pfkey}.
In Section~\ref{sec:lldpf}, we complete the proof of Theorem~\ref{thm:lld}.
In Section~\ref{sec:abs}, we state and prove an abstract version, Theorem~\ref{thm:abs}, of our main result, giving an LLD for a general class of nonuniformly hyperbolic systems modelled by Young towers with exponential tails.

\vspace{-2ex}
\paragraph{Notation}
We use ``big O'' and $\ll$ notation interchangeably, writing $b_n=O(c_n)$ or $b_n\ll c_n$
if there are constants $C>0$, $n_0\ge1$ such that
$b_n\le Cc_n$ for all $n\ge n_0$.
As usual, $b_n=o(c_n)$ means that $\lim_{n\to\infty}b_n/c_n=0$
and $b_n\sim c_n$ means that $\lim_{n\to\infty}b_n/c_n=1$.

We write $B_r(x)$ to denote the open ball in $\R^d$ and $\C$ of radius $r$ centred at $x$.

\section{Set up}
\label{sec:setup}

Recall that  the Lorentz map $(\tT,\tM,\tilde\mu)$ is a $\Z^d$-cover (here $d=1,2$) of the dispersing billiard $(T,M,\mu)$ 
by the free flight function $\kappa:M\to\Z^d$.
The measures $\tilde\mu$ and $\mu$ are $\tT$-invariant (respectively $T$-invariant) and are the Liouville measures on $\tM$ and $M$, normalised so that $\mu$ is a probability measure. 

The dispersing billiard $(T,M,\mu)$ is the collision map on the billiard domain $Q=\T^2\setminus\Omega$ where $\T^2=\R^2/\Z^2$ and $\Omega$ is a finite union of convex obstacles with $C^3$ boundaries and nonvanishing curvature.
The two-dimensional phase space (position in $\partial\Omega$ and unit velocity) is given by $M=\partial\Omega\times (-\pi/2,\pi/2)$.

Let $\tQ$ denote the lifted domain inside $\tM$.
Then $Q =\tQ/\Z^2\subset \T^2$ if $d=2$ and $Q
=\tQ/(\Z\times\{0\})\subset \T^2$ if $d=1$.

An important part of the proof of Theorem~\ref{thm:lld} is that $(T,M,\mu)$ is modelled by
a two-sided Young tower $(f,\Delta,\mu_\Delta)$ with exponential tails~\cite{Chernov99,Young98}.
We briefly recall the notion of Young tower.\footnote{We suppress many standard details about Young towers, mentioning only those aspects required for this paper. For instance, we suppress the fact that the projection $\bar\pi:\Delta\to\bar\Delta$ corresponds in practice to collapsing stable leaves.}

Let $(Y,\mu_Y)$ be a probability space with an at most countable measurable partition~$\alpha$, and let $F:Y\to Y$ be an ergodic measure-preserving transformation.
For $\theta\in(0,1)$, define the {\em separation time} $s(y,y')$ to be the least integer $n\ge0$ such that $F^ny$ and $F^ny'$ lie in distinct partition elements in $\alpha$.
It is assumed that the partition~$\alpha$ separates trajectories, so $s(y,y')=\infty$ if and only if $y=y'$;  then $d_\theta(y,y')=\theta^{s(y,y')}$ is a metric.
We say that $F$ is a {\em (full-branch) Gibbs-Markov map} if
\begin{itemize}

\parskip=-2pt
\item $F|_a:a\to Y$ is a measurable bijection for each $a\in\alpha$, and
\item
There are constants $C>0$, $\theta\in(0,1)$ such that
$|\log\xi(y)-\log\xi(y')|\le Cd_\theta(y,y')$ for all $y,y'\in a$, $a\in\alpha$, where
$\xi=\frac{d\mu_Y}{d\mu_Y\circ F}:Y\to\R$.
\end{itemize}

Let $F:Y\to Y$ be a Gibbs-Markov map and let
$\sigma:Y\to\Z^+$ be constant on partition elements such that 
$\mu_Y(\sigma>n)=O(e^{-an})$ for some $a>0$,
We define the {\em one-sided Young tower with exponential tails} $\bar\Delta=Y^\sigma$ and {\em tower map} $\bar f:\bar\Delta\to\bar\Delta$ as follows:
\[
\bar\Delta=\{(y,\ell)\in Y\times\Z: 0\le \ell\le \sigma(y)-1\},
\quad
\bar f(y,\ell)=\begin{cases}
(y,\ell+1) & \ell\le\sigma(y)-2 \\
(Fy,0) & \ell=\sigma(y)-1
\end{cases}.
\]
Let $\bar\sigma=\int_Y\sigma\,d\mu_Y$.
Then $\bar\mu_\Delta=(\mu_Y\times{\rm counting})/\bar\sigma$ is an ergodic $\bar f$-invariant probability measure on $\bar\Delta$.

We say that $(T,M,\mu)$ is modelled by a Young tower $(f,\Delta,\mu_\Delta)$ with exponential tails if there exist a one-sided Young tower
$(\bar f,\bar\Delta,\bar\mu_\Delta)$ and measure-preserving semiconjugacies 
\[
\pi:\Delta\to M, \qquad \bar\pi:\Delta\to\bar\Delta.
\]

Next, we recall some properties proved in~\cite{SV07} of
the free-flight function
$\kappa:M\to\Z^d$ between collisions.
First, there is a constant $C>0$ such that
$\mu(|\kappa|>n)\sim Cn^{-2}$.
Second, $\kappa$
 lifts to a function
$\hat\kappa=\kappa\circ\pi:\Delta\to\Z^d$ that is constant on 
$\bar\pi^{-1}(a\times\{\ell\})$ for each $a\in\alpha$, $\ell\in\{0,\dots,\sigma(a)\}$.  Hence $\hat\kappa$ projects to an observable $\bar\kappa:\bar\Delta\to\Z^d$ constant on the partition elements $a\times\{\ell\}$ of $\bar\Delta$.
In particular, $\bar\mu_\Delta(|\bar\kappa|>n)=\mu(|\kappa|>n)
\sim C n^{-2}$.

Define 
\[
\textstyle \psi:Y\to\R, \qquad \psi(y)=\sum_{\ell=0}^{\sigma-1}|\hat\kappa(y,\ell)|.
\]

\begin{prop}  \label{prop:psi}
There exists $C>0$ such that
\[
\mu_Y(\psi>n)\le C n^{-2}
\quad\text{for all $n\ge 1$.}
\]
In particular,
$\psi\in L^r(Y)$ for all $r<2$.
\end{prop}

\begin{proof}
This is proved in~\cite{SV07}.
The main step~\cite[Lemma~16]{SV07} uses the
bound $\mu(|\kappa|>n)=O(n^{-2})$ together with the structure of infinite horizon dispersing billiards
(see also~\cite[Lemma~5.1]{ChernovZhang08}). 
The bound for $\mu_Y(\psi>n)$ then follows (see for instance~\cite[Section~2]{ChernovZhang08})).~
\end{proof}

We end this subsection by recalling some results about transfer operators and perturbed transfer operators on the one-sided tower.
Let $P:L^1(\bar\Delta)\to L^1(\bar\Delta)$ be the transfer operator for $(\bar f,\bar\Delta,\bar\mu_\Delta)$, so $\int_{\bar\Delta} Pv\,w\,d\bar\mu_\Delta=\int_{\bar\Delta} v\,w\circ \bar f\,d\bar\mu_\Delta$ for all $v\in L^1$, $w\in L^\infty$. 
By~\cite[Section~3.3]{BalintGouezel06}, there is a Banach space $\cB'\subset L^1$ (called $\cH$ in~\cite{BalintGouezel06}) such that
$P:\cB'\to\cB'$ is quasicompact. (The definition of $\cB'$ is not used in this paper.) 
In particular, the intersection of the spectrum of $P:\cB'\to\cB'$ with the unit circle consists of finitely many eigenvalues 
$\lambda_0,\dots,\lambda_{q-1}$ of finite multiplicity
and these are the $q$'th roots of unity $\lambda_k=e^{2\pi ik/q}$.  
By ergodicity, these eigenvalues are simple.

We consider the perturbed family of transfer operators
\[
P_t:L^1(\bar\Delta)\to L^1(\bar\Delta), \quad
P_tv=P(e ^{i t\cdot \bar\kappa}v),\;t\in\R^d,
\]
 where $\,\cdot\,$ denotes the standard scalar product on $\R^d$. 
Applying results of~\cite{KellerLiverani99}, it is shown in~\cite[Section~3.3.2]{BalintGouezel06} that there exists $\delta>0$ so that
$t\mapsto P_t:\cB'\to L^3$ is continuous for $t\in B_\delta(0)$.
Moreover, there are continuous families of simple isolated eigenvalues $t\mapsto \lambda_{k,t}$ for $P_t:\cB'\to\cB'$ with $\lambda_{k,0}=\lambda_k$ and $|\lambda_{k,t}|\le1$.
Let $t\mapsto \Pi_{k,t}$ denote the corresponding spectral projections on $\cB'$.  Then
\begin{equation}
\label{eq:Sz}
P_t^n=\sum_{k=0}^{q-1} \lambda_{k,t}^n\Pi_{k,t}+Q_t^n,
\end{equation}
where $Q_t=P_t\big(I-\Pi_{0,t}-\dots-\Pi_{q-1,t}\big)$.
By~\cite[Corollary~2]{KellerLiverani99}, there exist $C>0$ and $\gamma\in(0,1)$ such that
\begin{equation}
\label{eq:Sz-bis}
\sup_{t\in B_\delta(0)}\| Q_t^n\|_{\mathcal B'}\le C\gamma^n.
\end{equation}

By~\cite{SV07},
\begin{equation} \label{eq:SV}
1-\lambda_{0,t}\sim \Sigma t\cdot t\, \log (1/|t|) 
\quad\text{as $t\to0$},
\end{equation}
where $\Sigma\in\R^{d\times d}$ is a positive-definite  matrix.

\section{The range $n\ll \log |N|$.}
\label{sec:bigN}

In this section, we prove Theorem~\ref{thm:lld} in the range $n\ll\log|N|$. This estimate holds at the level of $T:M\to M$ and $\kappa:M\to\Z^d$ (without requiring consideration of Young towers).
Recall that $d\in\{1,2\}$.

\begin{lemma} \label{lem:bigN}
Let $\omega>0$.  There exists $C>0$ such that
\[
\mu(\kappa_n=N)\le 
C\frac{n}{a_n^d}\frac{(\log|N|\log\log|N|)^{d/2}}{|N|^2}
\quad\text{for all $n\ge1$, $N\in\Z^d$ with $n\le \omega\log|N|$.}
\]
\end{lemma}

\begin{proof}
Let 
\[
S_n=\sum_{j=0}^{n-1}|\kappa|\circ T^j,
\qquad M_n=\max_{j=0,\dots,n-1}|\kappa|\circ T^j.
\]
Note that for $q\ge1$,
\[
\mu(M_n>q)\le \sum_{j=0}^{n-1}\mu(|\kappa|\circ T^j>q)
=n\mu(|\kappa|>q)\ll n/q^2.
\]
Hence
\begin{align} \nonumber
\mu(\kappa_n= N)\le \mu(S_n\ge |N|)
& \le \mu\big(M_n> \tfrac{|N|}{2}\big)+\mu\big(S_n\ge |N|,\,M_n\le \tfrac{|N|}{2}\big)
\\ & \ll n/|N|^2+\mu\big(S_n\ge |N|,\,M_n\le \tfrac{|N|}{2}\big).
\label{eq:SM}
\end{align}

Suppose that $S_n\ge|N|$ and $M_n\le |N|/2$.
There exists $i\in\{0,\dots,n-1\}$ such that
$|\kappa|\circ T^i\ge |N|/n$.
Then $\sum_{0\le j\le n-1,\,j\neq i}|\kappa|\circ T^j\ge |N|/2$
so there exists $j\in\{0,\dots,n-1\}\setminus\{i\}$ such that $|\kappa|\circ T^j\ge |N|/(2(n-1))$.
We can choose $c>0$ so that $|N|/(2(n-1))\ge cN^{4/5}$.
It follows that
\begin{align*}
\mu\big(S_n\ge |N|,\,M_n\le \tfrac{|N|}{2}\big)
 &\le \sum_{i,j=0,\dots,n-1,\,i\neq j}
\mu\big(|N|\ge |\kappa|\circ T^i\ge \tfrac{|N|}{n},\,|\kappa|\circ T^j\ge c|N|^{4/5}\big)
\\ & = \sum_{i,j=0,\dots,n-1,\,i\neq j}
\mu\big(|N|\ge |\kappa|\ge \tfrac{|N|}{n},\,|\kappa|\circ T^{j-i}\ge c|N|^{4/5}\big)
\\ & \le  n\sum_{|r|=1,\dots,n-1}
\mu\big(|N|\ge |\kappa|\ge \tfrac{|N|}{n},\,|\kappa|\circ T^r\ge c|N|^{4/5}\big)
\\ & = n\sum_{|N|\ge p\ge |N|/n}\sum_{1\le |r|<n}
\mu\big(|\kappa|=p,\,|\kappa|\circ T^r\ge c|N|^{4/5}\big)
\\ & \le  n\sum_{p\ge |N|/n}\sum_{1\le |r|<n}
\mu\big(|\kappa|=p,\,|\kappa|\circ T^r\ge cp^{4/5}\big).
\end{align*}
Note that the constraints $p\ge |N|/n$, $n\le \omega\log|N|$ imply that
\[
n\ll \log|N|\le \log(np) = \log n + \log p
\ll n/2 + \log p,
\]
so there is a constant $\omega'>0$ such that $n\le \omega'\log p$.
Hence
\[
\mu\big(S_n\ge |N|,\,M_n\le \tfrac{|N|}{2}\big)
\le n\sum_{p\ge |N|/n}\sum_{1\le |r|<\omega'\log p}
\mu\big(|\kappa|=p,\,|\kappa|\circ T^r\ge cp^{4/5}\big).
\]
By~\cite[Lemma~16]{SV07} (see also~\cite[Lemma~5.1]{ChernovZhang08}), there is a constant $C>0$ such that
\[
\mu\big(|\kappa|=p,\,|\kappa|\circ T^r\ge cp^{4/5}\big)\le Cp^{-2/45}\mu(|\kappa|=p)\ll p^{-(3+2/45)}
\]
for $1\le |r|<\omega'\log p$.
Hence taking $\eta=1/45$,
\begin{align*}
\mu\big(S_n\ge |N|,\,M_n\le \tfrac{|N|}{2}\big)
& \ll n\sum_{p\ge |N|/n}(\log p)\,p^{-(3+2/45)}
\ll n\sum_{p\ge |N|/n}p^{-(3+\eta)}
\\ & \ll \frac{n^{3+\eta}}{|N|^{2+\eta} }
=\frac{n}{|N|^2}\frac{n^{2+\eta}}{|N|^\eta}
\ll \frac{n}{|N|^2}.
\end{align*}
Combining this with~\eqref{eq:SM},
\[
\mu\big(\kappa_n= N)
\ll \frac{n}{|N|^2}= \frac{n}{a_n^d|N|^2}a_n^d.
\]
Finally,
\[
a_n=(n\log n)^{1/2}\ll (\log |N|\log\log |N|)^{1/2},
\]
completing the proof.
\end{proof}

\section{Key estimates on the one-sided tower}
\label{sec:key}

To prove Theorem~\ref{thm:lld}, it remains
by Lemma~\ref{lem:bigN}
to consider the range $\log|N|\le\eps_1 n$ where $\eps_1$ is chosen sufficiently small.
Since $\mu(\kappa_n=N)=\bar\mu_\Delta(\bar\kappa_n=N)$, it suffices to work on the one-sided tower $\bar\Delta$.
To simplify the notation, we write $(f,\Delta,\mu_\Delta)$ for the one-sided tower map, and $\kappa:\Delta\to\Z^d$ for the free flight function on the one-sided tower.

As clarified in~\cite[Lemma~5.1]{PT20},  the derivative of $P_t$ at $t=0$ is not a bounded operator from $\cB'\to L^1$.
In Section~\ref{sec:pfkey}, we define a more suitable Banach space $\cB \subset\cB'\cap L^\infty$ such that we have better control on $\Pi_t:\cB\to L^1$.
To apply the method from~\cite{MTsub}, we require the following lemmas. 
Let $\partial_j=\partial_{t_j}$ for $j=1,\dots,d$.
For $t,h\in\R^d$, $b>0$, set 
\[
  M_b(t,h)=|h|L(h)\big\{1+ L(h)\,|t|^2L(t) +|h|^{-b|t|^2L(t)}L(h)^2\,|t|^4L(t)^2\big\}
\]
where $L(t)=\log(1/|t|)$.

\begin{lemma}
\label{lem:key} 
Let $j\in\{1,\dots,d\}$, $k\in\{0,\dots,q-1\}$.
There exists $\delta>0$ such that 
$t\mapsto \lambda_{k,t}$ and $t\mapsto\Pi_{k,t}:\cB\to L^1$ are $C^1$ on $B_\delta(0)$. Moreover, $\partial_j\lambda_{k,0}=0$.

Furthermore, 
there exist $C>0$, $\delta>0$, $b>0$ such that 
for all $t,h\in B_\delta(0)$, 
\[
|\partial_j\lambda_{k,t+h}-\partial_j\lambda_{k,t}|  \le CM_b(t,h),
 \qquad
\|\partial_j\Pi_{k,t+h}-\partial_j\Pi_{k,t}\|_{\cB\mapsto L^1}  \le 
  CM_b(t,h).
\]
\end{lemma}

\begin{lemma} \label{lem:key2}
There exists $b>0$ such that
$|\lambda_k-\lambda_{k,t}|\le b|t|^2L(t)$ 
for all $t\in B_\delta(0)$, $k=0,\dots,q-1$.
\end{lemma}

\begin{rmk} 
Using the asymptotic in~\eqref{eq:SV}, we obtain the stronger conclusion
$\lambda_k-\lambda_{k,t}\sim\lambda_k\Sigma t\cdot t\, L(t)$ as $t\to0$.
However, this stronger conclusion is not used in this paper and we consider the weaker conclusion since it generalises to the abstract setting of Theorem~\ref{thm:abs}.
\end{rmk}

\begin{cor} \label{cor:int}
Let $\beta\ge0$, $r\in\R$, $k=0,\dots,q-1$.
There exist $C>0$, $\delta>0$ such that
\[
\int_{B_{2\delta}(0)} |t|^\beta L(t)^r|\lambda_{k,t}|^n \,dt\le 
C\frac{(\log n)^r}{a_n^{d+\beta}}
\quad\text{for all $n\ge1$}.
\]
\end{cor}

\begin{proof}
By Lemma~\ref{lem:key2}, after shrinking $\delta$, 
\[
|\lambda_{k,t}|-1\in(-1,0], \quad \log|\lambda_{k,t}|\le |\lambda_{k,t}|-1\le
|\lambda_{k,t}-\lambda_k|\le  -b|t|^2 L(t)
\]
for all $t\in B_{2\delta}(0)$, and so
$|\lambda_{k,t}|\le \exp\{-b|t|^2 L(t)\}$.
The result follows from~\cite[Lemma~2.3]{MTsub}.
(The argument in~\cite{MTsub} uses that $a_n$ satisfies
$n\log a_n\sim a_n^2$, so $a_n\sim (\frac12 n\log n)^{1/2}$, which agrees with the definition of $a_n$ used here up to an inconsequential constant factor.)
\end{proof}

\section{Proof of Lemmas~\ref{lem:key} and~\ref{lem:key2}}
\label{sec:pfkey}

We continue to work on the one-sided tower $\Delta$.
Fix $\theta\in(0,1)$ and recall the definition of the metric $d_\theta$ on $Y$ from Section~\ref{sec:setup}.
We define
the Banach space $\cB=\cB(\Delta)$ of dynamically H\"older observables $v:\Delta\to\R$ with 
$\|v\|_{\cB}<\infty$, where
\[
\|v\|_{\cB}=\sup_{(y,\ell)\in\Delta} |v(y,\ell)|
+ \sup_{(y,\ell)\neq(y',\ell)} 
\frac{|v(y,\ell)-v(y',\ell)|}{d_\theta(y,y')}.
\]

In this section, we often write $\cB(\Delta)$ and $L^1(\Delta)$ for the function spaces on the Young tower $\Delta$, to distinguish them from related function spaces defined on the base~$Y$.

\subsection{Renewal operators}
\label{sec:Y}

Let $R:L^1(Y)\to L^1(Y)$ denote the transfer operator corresponding to the Gibbs-Markov map $F:Y\to Y$, 
so $\int_Y Rv\,w\,d\mu_Y=\int_Y v\,w\circ F\,d\mu_Y$ for all $v\in L^1$, $w\in L^\infty$.
For $y\in Y$ and $a\in\alpha$, let $y_a$ denote the unique preimage $y_a\in a$ such that $Fy_a=y$. 
Recall that $(Rv)(y)=\suma \xi(y_a)v(y_a)$ and that there is a constant $C>0$ such that
\begin{equation} \label{eq:GM}
0<\xi(y_a)\le C\mu_Y(a), \qquad
|\xi(y_a)-\xi(y'_a)|\le C\mu_Y(a)d_\theta(y,y'),
\end{equation}
for all $y,y'\in Y$, $a\in\alpha$.
(Standard references for properties of the transfer operator $R$ for a Gibbs-Markov map include~\cite{AD01,ADU93}.) 

Define the Banach space $\cB_1(Y)$ of observables $v:Y\to\R$ with 
$\|v\|_{\cB_1(Y)}<\infty$ where
$\|v\|_{\cB_1(Y)}=\|v\|_\infty+\sup_{y\neq y'}|v(y)-v(y')|/d_\theta(y,y')$.

\begin{prop} \label{prop:Ruv} There exists $C>0$ such that
$\|R(uv)\|_{\cB_1(Y)}\le C \|u\|_1\|v\|_{\cB_1(Y)}$
for all $u\in L^1(Y)$ constant on partition elements and all $v\in\cB_1(Y)$.
\end{prop}

\begin{proof}
Since $u$ is constant on partition elements, we write $u(a)=u|_a$.
By~\eqref{eq:GM},
\[
\|R(uv)\|_\infty \ll \suma \mu_Y(a)|u(a)|\,d\mu_Y\,\|v\|_\infty
=\|u\|_1\|v\|_\infty.
\]
Next, 
let $y,y'\in Y$. Then
$(R(uv))(y)-(R(uv))(y')=I_1+I_2$ where
\[
I_1  =\suma (\xi(y_a)-\xi(y'_a))u(a)v(y_a), \qquad
  I_2  = \suma \xi(y'_a)u(a)(v(y_a)-v(y'_a)).
\]
By~\eqref{eq:GM},
\begin{align*}
|I_1| & \ll \suma \mu_Y(a) d_\theta(y,y')|u(a)|\|v\|_\infty
= \|u\|_1\|v\|_\infty\,d_\theta(y,y'), \\
|I_2| & \ll \suma \mu_Y(a) |u(a)|\|v\|_{\cB_1(Y)}\,d_\theta(y_a,y'_a)
\le  \|u\|_1\|v\|_{\cB_1(Y)}\,d_\theta(y,y').
\end{align*}
Hence $|(R(uv))(y)-(R(uv))(y')|\ll \|u\|_1\|v\|_{\cB_1(Y)}\,d_\theta(y,y')$, and the result follows.
\end{proof}

For $z\in\C$ with $|z|\le1$ and $t\in\R^d$, define
\begin{align*}
\hR(z,t) & :L^1(Y)\to L^1(Y), \qquad \hR(z,t)v=R(e^{it\cdot\kappa_\sigma}z^\sigma v)
=\sum_{n=1}^\infty z^n R_{t,n}v
\end{align*}
where $R_{t,n}v=R(1_{\{\sigma=n\}}e^{it\cdot\kappa_\sigma}v)$
and $\kappa_\sigma(y)=\sum_{\ell=0}^{\sigma-1}\kappa(y,\ell)$.
Recall that $\psi(y)=\sum_{\ell=0}^{\sigma(y)}|\kappa(y,\ell)|$.
By Proposition~\ref{prop:psi}, $\kappa_\sigma,\,\psi\in L^r(Y)$ for all $r<2$.

\begin{prop} \label{prop:Rt}
There exists $\delta>0$ such that, regarded as operators on $\cB_1(Y)$,
\begin{itemize}
\item[(a)] $z\mapsto \hR(z,t)$
is analytic on $B_{1+\delta}(0)$ for all $t\in\R^d$;
\item[(b)] $(z,t)\mapsto (\partial_z^k\hR)(z,t)$ is $C^1$ on $B_{1+\delta}(0)\times\R^d$ for $k=0,1,2$;
\item[(c)]
$z\mapsto (\partial_j\hR)(z,t)$ is $C^1$ on $B_{1+\delta}(0)$ uniformly in $t\in\R^d$ for $j=1,\dots,d$.
\end{itemize}
\end{prop}

\begin{proof}
It suffices to show that there exist $a>0$, $C>0$ such that
\[
 \|R_{t,n}\|_{\cB_1(Y)}\le Ce^{-an},  \qquad
\|\partial_jR_{t,n}\|_{\cB_1(Y)}\le Ce^{-an}, 
\]
 for all $t\in\R^d$, $j=1,\dots,d$, $n\ge1$.

Since $\kappa_\sigma\in L^r(Y)$ for all $r<2$ and $\sigma$ has exponential tails,
there exists $a>0$ such that
$\|1_{\{\sigma=n\}} \kappa_\sigma\|_1\ll e^{-an}$.

Note that $\sigma$ and $\kappa_\sigma$ are constant on partition elements.
By Proposition~\ref{prop:Ruv},
$\|R_{t,n}\|_{\cB_1(Y)}\ll \|1_{\{\sigma=n\}}\|_1$. 
Also,
\(
\|\partial_j R_{t,n}\|_{\cB_1(Y)}
  \ll \|1_{\{\sigma=n\}}\kappa_\sigma\|_1
\)
completing the proof.
\end{proof}

For $z\in\C$ with $|z|\le1$ and $t\in\R^d$, define
\begin{align*}
\hA(z,t) & :L^1(Y)\to L^1(\Delta), \qquad \hA(z,t)v
=\sum_{n=1}^\infty z^n A_{t,n}v
\end{align*}
where 
\(
(A_{t,n}v)(y,\ell) =1_{\{\ell=n\}}(P_t^nv)(y,\ell)
=1_{\{\ell=n\}} e^{it\cdot\kappa_n(y,0)}v(y).
\)

\begin{prop} \label{prop:A}
There exists $\delta>0$ such that 
regarded as operators from $L^\infty(Y)$ to $L^1(\Delta)$,
\begin{itemize}
\item[(a)] $z\mapsto \hA(z,t)$
is analytic on $B_{1+\delta}(0)$ for all $t\in\R^d$;
\item[(b)] $(z,t)\mapsto (\partial_z\hA)(z,t)$ is $C^1$ on $B_{1+\delta}(0)\times\R^d$.
\end{itemize}
\end{prop}

\begin{proof} 
Let $\|\;\|$ denote $\|\;\|_{L^\infty(Y)\mapsto L^1(\Delta)}$.
As in the proof of Proposition~\ref{prop:Rt}, it suffices 
to obtain exponential estimates for 
$\|A_{t,n}\|$ and $\|\partial_j A_{t,n}\|$.

There exists $a>0$ such that
$\|1_{\{\sigma>n\}}\psi\|_{L^1(Y)}=O(e^{-an})$.
Now
$(A_{t,n}v)(y,\ell)=1_{\{\ell=n\}}e^{it\cdot\kappa_n(y)}v(y)$ so
\[
\|A_{t,n}\|\le \int_\Delta 1_{\{\ell=n\}}\,d\mu_\Delta \le \|1_{\{\sigma>n\}}\|_{L^1(Y)}.
\]
Similarly,
$\|\partial_j A_{t,n}\|  \le 
\int_\Delta 1_{\{\ell=n\}}\psi\,d\mu_\Delta \le \|1_{\{\sigma>n\}}\psi\|_{L^1(Y)}$ completing the proof.
\end{proof}

For $z\in\C$ with $|z|\le1$ and $t\in\R^d$, define
\begin{align*}
\hB(z,t) & :L^1(\Delta)\to L^1(Y), \qquad \hB(z,t)v
=\sum_{n=1}^\infty z^n B_{t,n}v
\end{align*}
where 
\[
B_{t,n}v =1_YP_t^n(1_{D_n}v), \qquad
D_n=\{(y,\sigma(y)-n):y\in Y,\,\sigma(y)>n\}.
\]

\begin{prop} \label{prop:B}
There exists $\delta>0$ such that 
regarded as operators from $\cB(\Delta)$ to $\cB_1(Y)$,
\begin{itemize}
\item[(a)] $z\mapsto \hB(z,t)$
is analytic on $B_{1+\delta}(0)$ for all $t\in\R^d$;
\item[(b)] $(z,t)\mapsto (\partial_z\hB)(z,t)$ is $C^1$ on $B_{1+\delta}(0)\times\R^d$.
\end{itemize}
\end{prop}

\begin{proof} 
Let $\|\;\|$ denote $\|\;\|_{\cB(\Delta)\mapsto \cB_1(Y)}$.
Again, it suffices to obtain exponential estimates for 
$\|B_{t,n}\|$ and
$\|\partial_j B_{t,n}\|$.

We can write
$B_{t,n}v=R(u_{t,n}v_n)$ where
\[
u_{t,n}(y)=1_{\{\sigma(y)>n\}}e^{it\cdot\kappa_n(y,\sigma(y)-n)},
\qquad v_n(y)=1_{\{\sigma(y)>n\}}v(y,\sigma(y)-n).
\]
Note that $u_{t,n}$ is constant on partition elements and
$\|v_n\|_{\cB_1(Y)}\le \|v\|_\cB $.
Also, there exists $a>0$ such that
$\|1_{\{\sigma>n\}}\psi\|_{L^1(Y)}=O(e^{-an})$.

By Proposition~\ref{prop:Ruv},
\[
\|B_{t,n}\|\ll \|u_{t,n}\|_{L^1(Y)}=\|1_{\{\sigma>n\}}\|_{L^1(Y)}.
\]
Similarly,
$\|\partial_j B_{t,n}\|  \ll \|1_{\{\sigma>n\}}\psi\|_{L^1(Y)}$
completing the proof.
\end{proof}

For $z\in\C$ with $|z|\le1$ and $t\in\R^d$, define
\begin{align*}
\hE(z,t) & :\cB(\Delta)\to L^1(\Delta), \qquad \hE(z,t)v
=\sum_{n=1}^\infty z^n E_{t,n}v
\end{align*}
where 
\(
(E_{t,n}v)(y,\ell) =1_{\{\ell>n\}}(P_t^nv)(y,\ell).
\)

\begin{prop} \label{prop:E}
There exists $\delta>0$ such that 
regarded as operators from $\cB(\Delta)$ to $L^1(\Delta)$,
\begin{itemize}
\item[(a)] $z\mapsto \hE(z,t)$
is analytic on $B_{1+\delta}(0)$ for all $t\in\R^d$;
\item[(b)] $(z,t)\mapsto \hE(z,t)$ is $C^0$ on $B_{1+\delta}(0)\times\R^d$;
\end{itemize}
\end{prop}

\begin{proof}
Let $\|\;\|$ denote $\|\;\|_{\cB(\Delta)\mapsto L^1(Y)}$.
It suffices to obtain an exponential estimate for
$\|E_{t,n}\|$.
But $(E_{t,n}v)(y,\ell)=1_{\{\ell>n\}}e^{it\cdot\kappa_n(y,\ell-n)}v(y,\ell-n)$,
so 
\[
\|E_{t,n}\|\le
\int_\Delta 1_{\{\ell>n\}}\,d\mu_\Delta
\le \|\sigma\,1_{\{\sigma>n\}}\|_{L^1(Y)} 
\ll e^{-\eps n}
\]
as required.
\end{proof}

\subsection{Further estimates}

In this subsection, we exploit the fact (Proposition~\ref{prop:psi}) that
$\mu_Y(\psi>n)=O(n^{-2})$.

\begin{prop} \label{prop:furtherR}
There exist $C>0$, $\delta>0$, $b>0$ such that
\[
\|\partial_j\partial_z \hR(z,t+h)
-\partial_j\partial_z \hR(z,t)\|_{\cB_1(Y)}
  \le C |h| L(h)^2\big\{1
+ |h|^{-b\log |z|}L(h)(|z|-1)\big\},
\]
for all
$t,h\in B_\delta(0)$, all $z\in\C$ with $1\le|z|\le 1+\delta$,
and all $j=1,\dots,d$.
\end{prop}

\begin{proof}
In this argument, we take $|x|=\max_{j=1,\dots,d}|x_j|$ on $\R^d$ so that $\psi$ is integer-valued. 
Now,
$\partial_j\partial_z \hR(z,t)v=iR((\kappa_\sigma)_j
e^{it\cdot\kappa_\sigma} \sigma z^{\sigma-1}v)$.
By Proposition~\ref{prop:Ruv},
\begin{align*}
\|\partial_j\partial_z \hR(z,t+h)
-\partial_j\partial_z \hR(z,t)\|_{\cB_1(Y)}
& \ll \int_Y |\kappa_\sigma||e^{ih\cdot\kappa_\sigma}-1|\sigma |z|^\sigma\,d\mu_Y
\\ & \le 2\int_Y \psi\min\{|h|\psi,1\}\sigma |z|^\sigma\,d\mu_Y
 =2 \sum_{m,n=1}^\infty r_{m,n}
\end{align*}
where
\[
r_{m,n} =\mu_Y\big(\psi=m,\sigma=n\big ) mn \min\{|h|m,1\}|z|^n.
\]
Recall that $\mu_Y(\sigma=n)=O(e^{-an})$ for some $a>0$.
Fix $a_1\in(0,a)$ and $\delta>0$
so that $e^{-a}(1+\delta)< e^{-a_1}$.
Then 
\[
r_{m,n}\ll |h|m^2 ne^{-an}|z|^n\ll |h|m^2e^{-a_1n}.
\]
Fixing $b>0$ sufficiently large,
\[
\sum_{n>b\log m} r_{m,n} \ll |h|m^2 e^{-a_1b\log m}(1-e^{-a_1})^{-1}
\ll |h|m^2 m^{-a_1b}\le |h|m^{-2}.
\]
Hence
\(
\sum_{m=1}^\infty \sum_{n>b\log m} r_{m,n} \ll |h|.
\)

\vspace{1ex}
It remains to consider the terms with $n\le b\log m$.
Now
\[
|z|^n = 1+ (|z|^n-1) \le 1+ n|z|^{n-1}(|z|-1)
\ll 1 +(\log m) m^{b\log |z|}(|z|-1).
\]
Hence
\[
r_{m,n}\ll \mu_Y(\psi=m,\sigma=n)m(\log m)\min\big\{|h|m,1\}\{1+(\log m) m^{b\log |z|}(|z|-1)\big\}.
\]
and so
\[
\sum_{n\le b\log m}r_{m,n}\ll \mu_Y(\psi=m)m\min\big\{|h|m,1\}\{\log m+(\log m)^2 m^{b\log |z|}(|z|-1)\big\}.
\]
Let $K=[1/|h|]\ge1$.  Then for $m\le K$,
\[
\sum_{n\le b\log m}r_{m,n}\ll \mu_Y(\psi=m)|h|m^2\{\log K+(\log K)^2K^{b\log |z|}(|z|-1)\big\},
\]
and so by resummation,
\begin{align} \nonumber
\sum_{m=1}^K\sum_{n\le b\log m}r_{m,n} & 
\ll |h| \sum_{m=1}^K\mu_Y(\psi=m)m^2\{\log K+(\log K)^2K^{b\log |z|}(|z|-1)\big\}
\\ & \nonumber
\ll |h| (\log K)^2+ |h|(\log K)^3K^{b\log |z|}(|z|-1)
\\ & \ll |h| L(h)^2\big\{1+ |h|^{-b\log|z|}L(h)(|z|-1)\big\}.
\label{eq:est1}
\end{align}

Next,
\begin{align*}
\sum_{m>K}\sum_{n\le b\log m}r_{m,n} & 
\ll \sum_{m>K}\mu_Y(\psi=m)m(\log m)
\\ & \qquad\qquad + (|z|-1)\sum_{m>K}\mu_Y(\psi=m)m^{1+b\log|z|}(\log m)^2.
\end{align*}
Now,
\begin{align*}
\sum_{m>K} & \mu_Y(\psi=m)m^{1+b\log|z|}(\log m)^2
\\ & = 
\sum_{m>K}\mu_Y(\psi\ge m)m^{1+b\log|z|}(\log m)^2
-\sum_{m>K}\mu_Y(\psi> m)m^{1+b\log|z|}(\log m)^2
\\ & \le  \mu_Y(\psi> K)K^{1+b\log|z|}(\log K)^2
\\ & \qquad +
\sum_{m>K}\mu_Y(\psi\ge m)(m^{1+b\log|z|}(\log m)^2
-(m-1)^{1+b\log|z|}(\log (m-1))^2
\\ & \ll 
K^{b\log|z|-1}(\log K)^2+
(1+b\log|z|)
\sum_{m>K}\mu_Y(\psi\ge m) m^{b\log|z|}(\log m)^2
\\ & \ll 
|h|^{1-b\log|z|}L(h)^2+
\sum_{m>K} m^{b\log|z|-2}(\log m)^2.
\end{align*}
By Karamata,
\begin{align*}
\sum_{m>K} m^{b\log|z|-2}(\log m)^2
& \ll (1-b\log |z|)^{-1}K^{b\log|z|-1}(\log K)^2
\ll |h|^{1-b\log|z|}L(h)^2.
\end{align*}
Hence
\[
\sum_{m>K}\sum_{n\le b\log m} r_{m,n}
\ll |h|L(h)\big\{1+ |h|^{-b\log|z|}L(h)(|z|-1)\big\}.
\]
This combined with~\eqref{eq:est1} gives the desired estimate for
$\sum_{m\ge1}\sum_{n\le b\log m} r_{m,n}$, completing the proof.
\end{proof}

\begin{rmk} \label{rmk:furtherR}
Similarly,
\[
\|\partial_j\hR(z,t+h)
-\partial_j\hR(z,t)\|_{\cB_1(Y)}
  \le C |h| L(h)\big\{1
+ |h|^{-b\log |z|}L(h)(|z|-1)\big\}.
\]
\end{rmk}

\begin{prop} \label{prop:furtherA}
There exist $C>0$, $\delta>0$, $b>0$ such that
\[
\|\partial_j \hA(z,t+h)
-\partial_j \hA(z,t)\|_{\cB_1(Y)\mapsto L^1(\Delta)}
  \le C |h| L(h)\big\{1 + |h|^{-b\log |z|}L(h)(|z|-1)\big\},
\]
for all
$t,h\in B_\delta(0)$, all $z\in\C$ with $1\le|z|\le 1+\delta$,
and all $j=1,\dots,d$. 
\end{prop}

\begin{proof}
We have
\[
(\hA(z,t)v)(y,\ell)=\sum_{n=1}^\infty z^n1_{\{\ell=n\}}e^{it\cdot\kappa_n(y,0)}v(y)
=z^\ell e^{it\cdot\kappa_\ell(y,0)}v(y).
\]
Hence
\begin{align*}
\|\partial_j  \hA(z,t+h)- \partial_j  \hA(z,t)\|_{\cB_1(Y)\mapsto L^1(\Delta)}
& \ll \big\||z|^\sigma \psi\min\{|h|\psi,1\} \big\|_{L^1(Y)}
\\ & =\sum_{m,n=1}^\infty \mu_Y\big(\psi=m,\sigma=n\big)m|z|^n\min\{|h|m,1\}.
\end{align*}
We now proceed as in the proof of
Proposition~\ref{prop:furtherR}, except that
there is one less factor of $n$ (hence one less factor of $L(h)$).
\end{proof}

\begin{prop} \label{prop:furtherB}
There exist $C>0$, $\delta>0$, $b>0$ such that
\[
\|\partial_j \hB(z,t+h)
-\partial_j \hB(z,t)\|_{\cB(\Delta)\mapsto \cB_1(Y)}
  \le C |h| L(h) \big\{1
+ |h|^{-b\log |z|}L(h)(|z|-1)\big\},
\]
for all
$t,h\in B_\delta(0)$, all $z\in\C$ with $1\le|z|\le 1+\delta$,
and all $j=1,\dots,d$, 
\end{prop}

\begin{proof}
We have 
\begin{align*}
\|\partial_j  \hB(z,t+h)- \partial_j  \hB(z,t)\|_{\cB(\Delta)\mapsto \cB_1(Y)}
& \ll 
\sum_{n=1}^\infty |z|^n\|1_{\{\sigma>n\}}\psi\min\{h\psi,1\}\|_{L^1(Y)}.
\\ & =\sum_{m,n=1}^\infty \mu_Y\big(\psi=m,\sigma>n\big)m|z|^n\min\{|h|m,1\}.
\end{align*}
This is the same as in Proposition~\ref{prop:furtherA} except that $\sigma=n$ is replaced by $\sigma>n$ (which makes no difference given the exponential tails).
\end{proof}

\subsection{Spectral properties for $\hR(z,t)$} \label{sec:specR}

Throughout, we make use of the fact that $\lambda_k^\sigma=1$ (since $\sigma$ is divisible by $q$ and $\lambda_k$ is a $q$'th root of unity).
In particular, $\hR(\lambda_k,0)=\hR(1,0)=R$ for $k=0,\dots,q-1$.

\begin{prop} \label{prop:eig}
Let $z\in\C$, $|z|\le 1$. Then 
$1\in\spec \hR(z,0):\cB_1(Y)\to\cB_1(Y)$ if and only if $z^q=1$ in which case $1$ is a simple eigenvalue with eigenfunction $1$.
\end{prop}

\begin{proof} 
By quasicompactness, $1\in\spec \hR(z,0)$ if and only if $1$ is an eigenvalue.
Recall also that $1$ is a simple eigenvalue for $\hR(1,0)$ with eigenfunction $1$ and hence this also holds for $\hR(z,0)$ when $z^q=1$.

Finally, it is standard that 
$1$ is an eigenvalue for $\hR(z,0)$ if and only if
$\bar z$ is an eigenvalue for $P_0$ which, as noted in Section~\ref{sec:setup}, is the case if and only if $z^q=1$.  
\end{proof}

By Proposition~\ref{prop:Rt}(b), for $k=0,\dots,q-1$, the eigenvalue $1$ for $\hR(\bar\lambda_k,0)$ extends to a $C^1$ family
of simple isolated eigenvalues $(z,t)\mapsto \tau_k(z,t)$ on $B_\delta(\bar\lambda_k)\times B_\delta(0)$, for some $\delta>0$, with $\tau_k(\bar\lambda_k,0)=1$.
Let $\bar\sigma=\int_Y\sigma\,d\mu_Y$.
Recall that $L(t)=\log(1/|t|)$.

\begin{prop} \label{prop:mu0}
Let $0\le k,k'\le q-1$.
There are constants $C>0$, $\delta>0$ such that
\begin{itemize}
\item[(a)]
$|\tau_k(z,0)-1-\lambda_k\bar\sigma(z-\bar\lambda_k)|\le C|z-\bar\lambda_k|^2$ for $z\in B_\delta(\bar\lambda_k)$;
\item[(b)] $|\tau_k(\bar\lambda_k,t)-1|\le C|t|^2L(t)$ for $t\in B_\delta(0)$;
% \item[(c)]
% $|\partial_j\tau_k(\bar\lambda_k,t+h)-\partial_j\tau_k(\bar\lambda_k,t)|\le C |h|L(h)$
% for $t,h\in B_\delta(0)$, $j=1,\dots,d$;
\item[(c)] $|\tau_{k}(\bar\lambda_{k},t)-
\tau_{k'}(\bar\lambda_{k'},t)|\le C|t|^2$ for all
$t\in B_\delta(0)$.
\end{itemize}
\end{prop}

\begin{proof} 
Let $v_k(z,t)$ denote the $C^1$ families of eigenfunctions corresponding to the eigenvalues $\tau_k(z,t)$, with
\mbox{$v_k(\bar\lambda_k,0)=1$}. Normalise so that
\[
\int_Y \hR(\bar\lambda_k,0)v_k(z,t)\,d\mu_Y
=\int_Y v_k(z,t)\,d\mu_Y=1
\]
for $(z,t)\in B_\delta(\bar\lambda_k)\times B_\delta(0)$.
Then
\[
\tau_k(z,t)=\int_Y \hR(z,t)v_k(z,t)\,d\mu_Y=I_k(z,t)+J_k(z,t)
\]
where
\begin{align*}
I_k(z,t) & =\int_Y \hR(z,t)1\,d\mu_Y = \int_Y z^\sigma e^{it\cdot \kappa_\sigma}\,d\mu_Y, \\
J_k(z,t) &=
\int_Y (\hR(z,t)-\hR(\bar\lambda_k,0))(v_k(z,t)-v_k(\bar\lambda_k,0))\,d\mu_Y.
\end{align*}

Since $\hR$ and $v$ are $C^1$, it follows that
$J_k(z,0)=O(|z-\bar\lambda_k|^2)$ and $J_k(\bar\lambda_k,t)=O(|t|^2)$.
% Also,
% \[
% \partial_j J_k(\bar\lambda_k,t)=\int_Y\Big(\partial_j\hR(\bar\lambda_k,t)\big(v(\bar\lambda_k,t)-v(\bar\lambda_k,0)\big)
% + \big(\hR(\bar\lambda_k,t)-\hR(\bar\lambda_k,0)\big)\partial_j v(\bar\lambda_k,t)\Big)\,d\mu_Y.
% \]
% By Remark~\ref{rmk:furtherR},
% $|\partial_j J_k(\bar\lambda_k,t+h)-\partial_j J_k(\bar\lambda_k,t)| \ll |h|L(h)$.
Hence it suffices to consider the first term $I_k$.

For $t=0$, using that $\bar\lambda_k^\sigma=1$,
\begin{align*}
I_k(z,0)  =\int_Y(\bar\lambda_k+(z-\bar\lambda_k))^\sigma\,d\mu_Y
& =\int_Y (1+(\bar\lambda_k^{-1}z-1))^\sigma\,d\mu_Y
\\ & =1+\bar\lambda_k^{-1}\bar\sigma(z-\bar\lambda_k)+O(|z-\bar\lambda_k|)^2
\end{align*}
yielding part~(a).

Since $\int_Y\kappa_\sigma\,d\mu_Y=0$, for $z=1$, 
\[
I_k(\bar\lambda_k,t)
=1 +\int_Y (e^{it\cdot\kappa_\sigma}-1-it\cdot\kappa_\sigma)\,d\mu_Y,
\]
so
\begin{align*}
|I_k(\bar\lambda_k,t) -1| & \le \int_Y |e^{it\cdot \kappa_\sigma}-1-it\cdot\kappa_\sigma|\,d\mu_Y 
 \le 2\int_Y \min\{|t|^2\psi^2,|t|\psi\}\,d\mu_Y 
\\ & \le 2|t|^2 \sum_{0\le m\le 1/|t|} m^2\mu_Y(\psi=m)
+2|t| \sum_{m> 1/|t|} m\mu_Y(\psi=m).
\end{align*}
Using the tail estimate $\mu_Y(\psi>n)=O(n^{-2})$ and resummation
we obtain $|I_k(\bar\lambda_k,t)-1|\ll |t|^2L(t)$ proving (b).
% 
% Next,
% \begin{align*}
% |\partial_j I_k(\bar\lambda_k,t+h)-\partial_j I_k(t)|
% & \le \int_Y|\kappa_\sigma||e^{ih\cdot\kappa_\sigma}-1|\,d\mu_Y
% \\ & \le 2\int_Y \min\{|h||\kappa_\sigma|^2,|\kappa_\sigma|\}\,d\mu_Y
% \ll |h|L(h),
% \end{align*}
% completing the proof of part~(c).
% 
Finally, $I_k(\bar\lambda_k,t)$ is independent of $k$ yielding part~(c).
\end{proof}

It follows from Proposition~\ref{prop:mu0}(a) that $(\partial_z\tau_k)(\bar\lambda_k,0)>0$.
By the implicit function theorem, we can solve uniquely the equation $\tau_k(z,t)=1$ near $(\bar\lambda_k,0)$ to obtain a $C^1$ solution $z=g_k(t)$, $g_k:B_\delta(0)\to\C$, with $g_k(0)=\bar\lambda_k$.

\vspace{1ex}
Recall that
\(
M_b(t,h)=|h|L(h)\big\{1+ L(h)\,|t|^2L(t) +|h|^{-b|t|^2L(t)}L(h)^2\,|t|^4L(t)^2\big\}.
\)

\begin{cor} \label{cor:g}
There exist $C>0$, $\delta>0$, $b>0$ such that
for all $t,\,h\in B_\delta(0)$, $j=1,\dots,d$, $k=0,\dots,q-1$,
\begin{itemize}
\item[(a)]
$|g_k(t)-\bar\lambda_k|\le C|t|^2L(t)$; 
\item[(b)]
$|\partial_j g_k(t+h)-\partial_j g_k(t)|\le CM_b(t,h)$.
\end{itemize}
\end{cor}

\begin{proof}
Write
\begin{equation} \label{eq:mu}
\tau_k(z,t) =\tau_k(\bar\lambda_k,t)+(z-\bar\lambda_k)c_k(z,t).
\end{equation}
It follows from Proposition~\ref{prop:Rt}(b) that
$(z,t)\mapsto \partial_z\tau_k(z,t)$ is $C^1$.
Introducing momentarily the function $\zeta(s)=\tau_k(\bar\lambda_k+s(z-\bar\lambda_k),t)$, 
\begin{equation} \label{eq:c}
c_k(z,t)=(z-\bar\lambda_k)^{-1}\int_0^1 \zeta'(s)\,ds=\int_0^1 (\partial_z\tau_k)(1+s(z-\bar\lambda_k),t)\,ds.
\end{equation}
We deduce that
$(z,t)\mapsto c_k(z,t)$ is $C^1$.
By Proposition~\ref{prop:mu0}(a), $|c_k(\bar\lambda_k,0)|=\bar\sigma>0$ and we can shrink $\delta$ if necessary so that $c_k(\bar\lambda_k,t)$ is bounded away from zero for $t\in B_\delta(0)$.

Solving $\tau_k(z,t)=1$, 
\begin{equation} \label{eq:g}
g_k(t)-\bar\lambda_k=z-\bar\lambda_k\sim c_k(\bar\lambda_k,t)^{-1}(1-\tau_k(\bar\lambda_k,t)).
\end{equation}
The spectral radius of $\hR(\bar\lambda_k,t)$ is at most $1$ for all $t$, so $\tau_k(\bar\lambda_k,t)\in B_1(0)$.
Hence $|g_k(t)|\ge1$ for all~$t$.
By Proposition~\ref{prop:mu0}(b),
\[
|g_k(t)-\bar\lambda_k|\sim |c_k(\bar\lambda_k,t)|^{-1}|1-\tau_k(\bar\lambda_k,t)|\ll |t|^2L(t).
\]
proving part (a).

Implicit differentiation of $\tau_k(g_k(t),t)\equiv1$ yields
\[
\partial_j g_k(t)=-\partial_j \tau_k(g_k(t),t)/\partial_z \tau_k(g_k(t),t).
\]
By smoothness of $\partial_z\tau_k$ and $g$, the denominator $t\mapsto \partial_z \tau_k(g_k(t),t)$ is $C^1$.
We claim that
\begin{equation} \label{eq:tauk}
|\partial_j \tau_k(g_k(t+h),t+h)  -\partial_j \tau_k(g_k(t),t)|
  \ll M_b(t,h)
\end{equation}
from which part~(b) follows.

It follows from Proposition~\ref{prop:Rt}(c) that $z\mapsto\partial_j\tau_k(z,t)$ is $C^1$ uniformly in $t$. Also, $g$ is $C^1$,
so $|\partial_j \tau_k(g_k(t+h),t+h)-\partial_j \tau_k(g_k(t),t+h)|\ll |h|$.
By~\eqref{eq:mu},
\begin{align*}
|\partial_j \tau_k(z,t+h)- & \partial_j \tau_k(z,t)|
 \\ & \le
 |\partial_j\tau_k(\bar\lambda_k,t+h)-\partial_j\tau_k(\bar\lambda_k,t)|+|z-\bar\lambda_k|\,|\partial_jc_k(z,t+h)-\partial_jc_k(z,t)|.
\end{align*}
By Remark~\ref{rmk:furtherR}, 
$|\partial_j \tau_k(\bar\lambda_k,t+h)-  \partial_j \tau_k(\bar\lambda_k,t)|\ll |h|L(h)$.
By~\eqref{eq:c} and Proposition~\ref{prop:furtherR},
\[
|\partial_jc_k(z,t+h)- \partial_jc_k(z,t)|
\ll |h|L(h)^2 \big\{1+
|h|^{-b\log|z|}L(h)|z-\bar\lambda_k|\big\}.
\]
Hence
\[
|\partial_j \tau_k(z,t+h)-\partial_j \tau_k(z,t)|
 \ll |h|L(h)+
|h|L(h)^2|z-\bar\lambda_k|+
|h|^{1-b\log|z|}L(h)^3|z-\bar\lambda_k|^2,
\]
for $|z|\ge1$.
But $|g_k(t)|\ge1$, so by part~(a), 
\[
|\partial_j \tau_k(g_k(t),t+h)-\partial_j \tau_k(g_k(t),t)|
  \ll M_b(t,h)
\]
completing the proof of the claim.
\end{proof}

Let $\pi_k(z,t):\cB_1(Y)\to\cB_1(Y)$ denote the spectral projection corresponding to $\tau_k(z,t)$. 

\begin{lemma} \label{lem:mu}
There exists $\delta>0$ such that
\begin{equation} \label{eq:H}
(1-\tau_k(z,t))^{-1}\pi_k(z,t)=(g_k(t)-z)^{-1}\tilde\pi_k(t)+H_k(z,t)
\end{equation} 
where $\tilde\pi_k(t),\,H_k(z,t):\cB_1(Y)\to \cB_1(Y)$ are families of bounded operators satisfying
\begin{itemize}
\item[(a)] $\tilde\pi_k$ is $C^1$ on $B_\delta(0)$;
\item[(b)] $H_k$ is $C^0$ on $B_\delta(\bar\lambda_k)\times B_\delta(0)$;
\item[(c)]
$z\mapsto H_k(z,t)$ is analytic on $B_\delta(\bar\lambda_k)$ for $t\in B_\delta(0)$.
\end{itemize}
Moreover, there are constants $C>0$, $b>0$ such that $|\partial_j\tilde\pi_k(t+h)-\partial_j\tilde\pi_k(t)|\le CM_b(t,h)$ 
for $t,h\in B_\delta(0)$, $j=1\ldots, d$.
\end{lemma}

\begin{proof}
Fix $j$ and $k$.
Throughout this proof, we use the following abbreviations (for $r\ge0$):
\begin{itemize}

\parskip=-2pt
\item[(a)] ``$C^r$ uniformly in $z$'' means $C^r$ on $B_\delta(0)$ uniformly in $z\in B_\delta(\bar\lambda_k)$;
\item[(b)] ``jointly $C^r$'' means $C^r$ on $B_\delta(\bar\lambda_k)\times B_\delta(0)$;
\item[(c)] ``analytic'' means analytic on $B_\delta(\bar\lambda_k)$ for all $t\in B_\delta(0)$.
\end{itemize}

\noindent {\bf Step 1}
Write
\[
\pi_k(z,t)=\pi_k(g_k(t),t)+(g_k(t)-z)\tH(z,t).
\]
It follows from Proposition~\ref{prop:Rt}(a) that 
$\pi_k$ is analytic and hence that $\tH$ is analytic.

Next, 
$\partial_j(\pi_k(g_k(t),t))=G_1(t)+G_2(t)$
where
\[
G_1(t)=(\partial_z\pi_k)(g_k(t),t)\cdot \partial_j g_k(t),
\qquad
G_2(t)=(\partial_j\pi_k)(g_k(t),t). 
\]
It follows from Proposition~\ref{prop:Rt}(b) that
$\partial_z\pi_k$ is jointly $C^1$.
Also, $g$ is $C^1$.
Hence, by Corollary~\ref{cor:g}(b),
\[
|G_1(t+h)-G_1(t)|\ll |h|+|\partial_jg_k(t+h)-\partial_jg_k(t)|\ll M_b(t,h).
\]
Next, we note that $G_2$, with $\pi_k$ changed to $\tau_k$, was estimated in~\eqref{eq:tauk}, and the identical argument shows that
\(
|G_2(t+h)-G_2(t)|\ll M_b(t,h).
\)
Hence
\[
|\partial_j(\pi_k(g_k(t+h),t+h))-
\partial_j(\pi_k(g_k(t),t))|\ll M_b(t,h).
\]

Writing
$\tH(z,t)=\int_0^1(\partial_z\pi_k)((1-s)z+sg_k(t),t)\,ds$, we obtain that
$\tH$ is jointly $C^0$.

\vspace{2ex}
\noindent {\bf Step 2}
Write
\[
1-\tau_k(z,t)= \tau_k(g_k(t),t)-\tau_k(z,t)=(g_k(t)-z)\beta(z,t).
\]
Again, it follows from Proposition~\ref{prop:Rt}(a) that 
$\tau_k$ is analytic and hence that $b$ is analytic.
Also, 
it follows from Proposition~\ref{prop:Rt}(b) that 
$\partial_z^2\tau_k$ is jointly $C^1$.
Writing
$\beta(z,t)=\int_0^1 \partial_z\tau_k((1-s)z+sg_k(t),t)\,ds$, we obtain that
$\partial_z\beta$ is jointly $C^1$. By Proposition~\ref{prop:furtherR},
\[
|\partial_j \beta(g_k(t+h),t+h)-\partial_j \beta(g_k(t),t)|\ll M_b(t,h).
\]

By Proposition~\ref{prop:mu0}(a),  $|\beta(\lambda_k,0)|=\bar\sigma>0$ and we can shrink $\delta$ if necessary so that
$\beta$ is bounded away from zero on $B_\delta(\bar\lambda_k)\times B_\delta(0)$.
Let $\tilde \beta(z,t)=\beta(z,t)^{-1}$.
Then, we can write 
\[
(1-\tau_k(z,t))^{-1}=
(g_k(t)-z)^{-1}\big\{\tilde \beta(g_k(t),t)+(g_k(t)-z)q(z,t)\big\},
\]
where $q$ is analytic and jointly $C^0$ and
\[
|\partial_j \tilde \beta(g_k(t+h),t+h)- \partial_j \tilde \beta(g_k(t),t)|
\ll M_b(t,h).
\]

\vspace{2ex}
\noindent {\bf Step 3} Combining Steps 1 and 2, we obtain~\eqref{eq:H}
with
\begin{align*}
\tilde\pi_k(t) & =\tilde \beta(g_k(t),t)\pi_k(g_k(t),t), \\ 
H_k(z,t) & = 
q(z,t)\pi_k(g_k(t),t)+\tilde \beta(g_k(t),t)\tH(z,t)+(g_k(t)-z)q(z,t)\tH(z,t).
\end{align*}
The desired regularity properties of $\tilde\pi_k$ and $H_k$ follow immediately from the regularity properties established in Steps 1 and 2.
\end{proof}

\begin{cor} \label{cor:Rt}
There exists $\delta>0$ such that
\[
(I-\hR(z,t))^{-1}=\sum_{k=0}^{q-1}(g_k(t)-z)^{-1}\tilde\pi_k(t)+\hH(z,t),\quad
(z,t)\in B_{1+\delta}(0)\times B_\delta(0),
\]
where $\tilde\pi_k$ is as in Lemma~\ref{lem:mu} and
$\hH(z,t):\cB_1(Y)\to \cB_1(Y)$ is a family of bounded operators satisfying
\begin{itemize}
\item[(a)] $\hH$ is $C^0$ on $B_{1+\delta}(0)\times B_\delta(0)$;
\item[(b)]
$z\mapsto \hH(z,t)$ is analytic on $B_{1+\delta}(0)$ for $t\in B_\delta(0)$.
\end{itemize}
\end{cor}

\begin{proof}
Let $t\in B_\delta(0)$. For $z\in B_\delta(\bar\lambda_k)$, the spectrum of $\hR(z,t)$ is bounded uniformly away from $1$  except for the simple eigenvalue $\tau_k(z,t)$ near $1$.
Hence 
\[
(I-\hR)^{-1}  =(1-\tau_k)^{-1}\pi_k+\hH \quad
\text{on $B_\delta(\bar\lambda_k)\times B_\delta(0)$},
\]
where $z\mapsto\hH(z,t)$ is analytic on $B_\delta(\bar\lambda_k)$
for $t\in B_\delta(0)$ and
$\hH$ is $C^0$ on $B_\delta(\bar\lambda_k)\times B_\delta(0)$.
Applying Lemma~\ref{lem:mu} and relabelling,
\[
(I-\hR(z,t))^{-1} =(g_k(t)-z)^{-1}\tilde\pi_k(t)  +\hH(z,t).
\]
In addition, $z\mapsto (I-\hR(z,t))^{-1}$ is analytic on $B_{1+\delta}(0)\setminus \bigcup_k B_\delta(\bar\lambda_k)$ and
$(z,t)\mapsto (I-\hR(z,t))^{-1}$ is $C^0$ on $(B_{1+\delta}(0)\setminus \bigcup_k B_\delta(\bar\lambda_k))\times B_\delta(0)$. Hence, we obtain the desired result on $B_{1+\delta}(0)\times B_\delta(0)$.
\end{proof}

\subsection{Completion of the proof of Lemmas~\ref{lem:key} and~\ref{lem:key2}}
\label{sec:pfbc}

Define for $t\in\R^d$, $n\ge1$,
\begin{align*}
T_{t,n} & :L^1(Y)\to L^1(Y), \qquad T_{t,n}v=1_YP_t^n(1_Yv).
\end{align*}
For $z\in\C$, $t\in\R^d$, define
\(
\hP(z,t)=\sum_{n=0}^\infty z^n P_t^n, \,
\hT(z,t)=\sum_{n=0}^\infty z^n T_t^n.
\)
By~\cite{Sarig02}, we have the renewal equation $\hT=(I-\hR)^{-1}$.
Also, by~\cite{Gouezel05}, 
$\hP=\hA\hT\hB+\hE$.

Throughout, we work on the domain $B_{1+\delta}(0)\times B_\delta(0)\subset\C\times\R^d$.
Applying the renewal equation, Corollary~\ref{cor:Rt} becomes
\begin{equation}  \label{eq:T}
\hT(z,t)= \sum_{k=0}^{q-1} (g_k(t)-z)^{-1}\tilde\pi_k(t)+\hH(z,t),
\end{equation} 
where $\tilde\pi_k,\,\hH:\cB_1(Y)\to\cB_1(Y)$ are families of bounded operators satisfying:
$\tilde \pi_k$ is $C^1$;
$\hH$ is $C^0$; 
$z\mapsto \hH(z,t)$ is analytic for all $t$.
Moreover,
$|\partial_j\tilde\pi_k(t+h)-\partial_j\tilde\pi_k(t)|\ll M_b(t,h)$.

The same argument as in Step~1 of Lemma~\ref{lem:mu} (using 
Propositions~\ref{prop:B} and~\ref{prop:A} instead of Proposition~\ref{prop:Rt}) shows that
\begin{align} \label{eq:A}
\hA(z,t) &= \tA_k(t)+(g_k(t)-z)\hH_{k,1}(z,t), \qquad \tA_k(t)=\hA(g_k(t),t), \\
\hB(z,t) &= \tB_k(t)+(g_k(t)-z)\hH_{k,2}(z,t), \qquad \tB_k(t)=\hB(g_k(t),t),
\label{eq:B}
\end{align} 
where $\tA_k,\,\hH_{k,1}:\cB_1(Y)\to L^1(\Delta)$ 
and $\tB_k,\,\hH_{k,2}:\cB(\Delta)\to\cB_1(Y)$ 
are families of bounded operators satisfying:
$\tA_k,\,\tB_k$ are $C^1$;
$\hH_{k,r}$ is $C^0$; 
$z\mapsto \hH_{k,r}(z,t)$ is analytic for all $t$; for $r=1,2$.
Moreover, by Propositions~\ref{prop:furtherA} and~\ref{prop:furtherB},
\begin{align*}
\|\partial_h\tA_k(t+h)-\partial_j\tA_k(t)\|_{\cB(Y_1)\mapsto L^1(\Delta)}\ll M_b(t,h), \\
\|\partial_h\tB_k(t+h)-\partial_j\tB_k(t)\|_{\cB(\Delta)\mapsto \cB_1(Y)}\ll M_b(t,h). 
\end{align*}

Combining~\eqref{eq:T},~\eqref{eq:A} and~\eqref{eq:B} together with
Proposition~\ref{prop:E}, 
\[
\hP(z,t)=\hA(z,t)\hT(z,t)\hB(z,t)+\hE(z,t)
 =\sum_{k=0}^{q-1}\Big((g_k(t)-z)^{-1}\tilde\pi_{k,1}(t)+\hH_{k,3}(z,t)\Big),
\]
where $\tilde \pi_{k,1},\,\hH_{k,3}:\cB(\Delta)\to L^1(\Delta)$ are families of bounded operators satisfying:
$\tilde \pi_{k,1}$ is $C^1$;
$\hH_{k,3}$ is $C^0$; 
$z\mapsto \hH_{k,3}(z,t)$ is analytic for all $t$;
and 
$|\partial_j\tilde\pi_{k,1}(t+h)-\partial_j\tilde\pi_{k,1}(t)|\ll M_b(t,h)$.

Let $\|\;\|$ denote $\|\;\|_{\cB(\Delta)\mapsto L^1(\Delta)}$.
An immediate consequence of the regularity properties of $\hH_{k,3}$ is that
there exists $\gamma\in(0,1)$ such that
the Taylor coefficients of $\hH_{k,3}$ satisfy
\(
\|(H_{k,3})_{t,n}\|\ll \gamma^n.
\)
Hence,
\[
\Big\|P_t^n-\sum_{k=0}^{q-1}g_k(t)^{-(n+1)}\tilde\pi_{k,1}(t)\Big\| \ll \gamma^n.
\]
By~\eqref{eq:Sz} and~\eqref{eq:Sz-bis}, $\|P_t^n-\sum_{k=0}^{q-1}\lambda_{k,t}^n\Pi_{k,t}\|\ll\gamma^n$ for some $\gamma\in(0,1)$.
Altogether, we have shown that there exist $\gamma\in(0,1)$, $C>0$ such that
\[
\Big\|\sum_{k=0}^{q-1}\Big(\lambda_{k,t}^n\Pi_{k,t}- g_k(t)^{-(n+1)}\tilde\pi_{k,1}(t)\Big)\Big\|
\le C\gamma^n
\quad\text{for all $t\in B_\delta(0)$, $n\ge1$.}
\]

Since $\lambda_{k,0}=g_k(0)=1$, we can shrink $\delta>0$ so that $|\lambda_{k,t}|>\gamma$ and $|g_k(t)^{-1}|>\gamma$.
It follows that
$\lambda_{k,t}=g_k(t)^{-1}$ and
$\Pi_{k,t}=g_k(t)^{-1}\tilde\pi_{k,1}(t)$.
The desired regularity properties of $\lambda_{k,t}$ and $\Pi_{k,t}:\cB(\Delta)\mapsto L^1(\Delta)$ now follow from those for $g_k$ and $\tilde\pi_{k,1}$,
completing the proof of Lemma~\ref{lem:key}.

\begin{pfof}{Lemma~\ref{lem:key2}}
By~\eqref{eq:SV},
\[
|g_0(t)-1|=|\lambda_{0,t}^{-1}|\,|1-\lambda_{0,t}|\ll |t|^2L(t).
\]
By~\eqref{eq:g} with $k=0$,
\[
|1-\tau_0(1,t)|\sim|c_0(1,t)|\,|g_0(t)-1|\ll |t|^2L(t).
\]
Hence by Proposition~\ref{prop:mu0}(c),
\[
|1-\tau_k(\bar\lambda_k,t)|\ll  |t|^2L(t)
\]
for all $k=0,\dots,q-1$. Applying~\eqref{eq:g} once more,
\[
|g_k(t)-\bar\lambda_k|\sim |c_k(\bar\lambda_k,t)^{-1}|\,|1-\tau_k(\bar\lambda_k,t)|
\ll |t|^2L(t).
\]
Finally,
\[
|\lambda_k-\lambda_{k,t}|=|g_k(t)^{-1}|\,|g_k(t)-\bar\lambda_k| \ll |t|^2L(t)
\]
completing the proof.
\end{pfof}

\section{Proof of the main result}
\label{sec:lldpf}

In this section, we complete the proof of Theorem~\ref{thm:lld}.
We continue to work on the one-sided tower $\Delta$.

Fix $\delta$ as in Section~\ref{sec:key}.
Let $r:\R^d\to\C$ be $C^2$ with $\supp r\subset B_\delta(0)$ and define
$A_{n,N}=\int_{\R^d} e^{-it\cdot N}r(t)P_t^n\,dt$.
By~\eqref{eq:Sz},
\[
A_{n,N}
=\sum_{k=0}^{q-1}\int_{B_\delta(0)} e^{-it\cdot N}r(t)\lambda_{k,t}^n\Pi_{k,t}\,dt+
\int_{B_\delta(0)} e^{-it\cdot N}r(t)Q_t^n\,dt.
\]
Following~\cite{MTsub}, the main step in the proof of Theorem~\ref{thm:lld}
is to estimate $\|A_{n,N}\|$.
Throughout this section, $\|\;\|$ denotes $\|\;\|_{\cB\mapsto L^1}$.

The next result suffices in the range $|N|\le  a_n$.

\begin{cor} \label{cor:smallN} There exists $C>0$ such that
$\|A_{n,N}\|\le C a_n^{-d}$ for all $n\ge1$, $N\in\Z^d$.
\end{cor}

\begin{proof} By~\eqref{eq:Sz-bis} and Corollary~\ref{cor:int},
\(
\|A_{n,N}\|\ll \sum_{k=0}^{q-1}\int_{B_\delta(0)}|\lambda_{k,t}|^n\,dt +\gamma^n
\ll a_n^{-d}.
\)
\end{proof}

Recall from the proof of Corollary~\ref{cor:int} that
there is a constant $c>0$ such that
$\log|\lambda_{k,t}|\le -c|t|^2L(t)$.
Let $b>0$ be as in Lemma~\ref{lem:key} and define
$\eps_1=c/(2b)$.
We now focus on the range
\[
a_n\le |N|\le e^{\eps_1 n}.
\]

Choose $j$ so that $|N_j|=\max\{|N_1|,\dots,|N_d|\}$ and
set $h=\pi N_j^{-1}e_j$ (where $e_j\in\R^d$ is the $j$'th canonical unit vector).

\begin{prop} \label{prop:worst}
There exist $C>0$, $\delta>0$ such that
\[
\int_{B_{2\delta}(0)}\big|\partial_j(\lambda^n)_{k,t}-\partial_j(\lambda^n)_{k,t-h}\big|\,dt
\le 
C\frac{n}{a_n^d}\frac{\log|N|}{|N|}
\]
for all $n\ge1$, $|N|> \pi/\delta$ with $a_n\le |N|\le e^{\eps_1 n}$,
$k=0\dots,q-1$.
\end{prop}

\begin{proof}
Set $s=t-h$ and relabel so that $|\lambda_{k,s}|\le|\lambda_{k,t}|$. Then
$\int_{B_{2\delta}(0)}|\partial_j(\lambda^n)_t-\partial_j(\lambda^n)_{k,s}|\,dt=J+K$
where
\[
J  =n\int_{B_{2\delta}(0)}|\lambda_{k,t}^{n-1}-\lambda_{k,s}^{n-1}||\partial_j\lambda_{k,t}|\,dt,
\qquad 
K  =n\int_{B_{2\delta}(0)}|\lambda_{k,s}|^{n-1}|\partial_j\lambda_{k,t}-\partial_j\lambda_{k,s}|\,dt.
\]
By Lemma~\ref{lem:key},
\begin{align} \label{eq:K}
K   \ll 
\frac{n\log|N|}{|N|}\int_{B_{2\delta}(0)}|\lambda_{k,t}|^{n}\,dt
& +
\frac{n(\log|N|)^2}{|N|}\int_{B_{2\delta}(0)}|\lambda_{k,t}|^{n}|t|^2L(t)\,dt
\\ & +
\frac{n(\log|N|)^3}{|N|}\int_{B_{2\delta}(0)}|\lambda_{k,t}|^{n}|N|^{b|t|^2L(t)}|t|^4L(t)^2\,dt.  \nonumber
\end{align}
Since 
\[
|N|^{b|t|^2L(t)}=e^{b(\log|N|)|t|^2L(t)}
\le e^{b\eps_1n|t|^2L(t)}
= e^{\frac12 cn|t|^2L(t)}\le \lambda_{k,t}^{-n/2},
\]
it follows from Corollary~\ref{cor:int} that
\[
 \int_{B_{2\delta}(0)}|\lambda_{k,t}|^{n}|N|^{b|t|^2L(t)}|t|^4L(t)^2\,dt
 \le \int_{B_{2\delta}(0)}|\lambda_{k,t}|^{n/2}|t|^4L(t)^2\,dt
 \ll \frac{(\log n)^2}{a_n^{d+4}}.
\]
The other integrals in~\eqref{eq:K} are also estimated using Corollary~\ref{cor:int} and we obtain
\begin{align*}
K & \ll \frac{n}{a_n^d} \frac{\log|N|}{|N|} \Big\{
1+ \frac{\log|N| \log n}{a_n^2}
+ \frac{(\log|N|)^2 (\log n)^2}{a_n^4}
\Big\}
% \\  & = \frac{n}{a_n^d} \frac{\log|N|}{|N|} \Big\{ 1+ \frac{\log|N|}{n} + \frac{(\log|N|)^2}{n^2} \Big\}
\ll  \frac{n}{a_n^d} \frac{\log|N|}{|N|} .
\end{align*}
(Here, we used that $\log|N|\ll n = a_n^2/\log n$.)

Next,
\(
|\lambda_{k,t}^{n-1}-\lambda_{k,s}^{n-1}|
\le n|\lambda_{k,t}|^{n-2}|\lambda_{k,t}-\lambda_{k,s}|
\)
so by the mean value theorem,
\[
|\lambda_{k,t}^{n-1}-\lambda_{k,s}^{n-1}|
\ll \frac{n}{|N|}|\lambda_{k,t}|^{n}|\partial_j\lambda_{k,u}|
\]
for some $u$ between $t$ and $s$.
By Lemma~\ref{lem:key},
\begin{align*}
& |\partial_j\lambda_{k,t}|
 =|\partial_j\lambda_{k,t}-\partial_j\lambda_{k,0}|
\ll M_b(0,t)=|t|L(t),
\\ & |\partial_j\lambda_{k,u}|
\ll |u|\log(1/|u|)\ll |t|L(t)+\tfrac{\log|N|}{|N|}.
\end{align*}
Hence by Corollary~\ref{cor:int},
\begin{align*}
J & \ll
\frac{n^2}{|N|}\int_{B_{2\delta}(0)}|\lambda_{k,t}|^{n}
|\partial_j\lambda_{k,t}| |\partial_j\lambda_{k,u}|\,dt
\\ &
\ll \frac{n^2}{|N|}\int_{B_{2\delta}(0)}|\lambda_{k,t}|^{n} |t|^2 L(t)^2\,dt
+ \frac{n^2\log|N|}{|N|^2}\int_{B_{2\delta}(0)}|\lambda_{k,t}|^{n}|t|L(t) \,dt
\\ & \ll
\frac{n}{a_n^d}\frac{\log|N|}{|N|}\Big(\frac{n\log^2 a_n}{a_n^2}\frac{1}{\log|N|}
+\frac{n\log a_n}{a_n}\frac{1}{|N|}\Big)
\\ & =
\frac{n}{a_n^d}\frac{\log|N|}{|N|}\Big(\frac{\log a_n}{\log|N|}
+\frac{a_n}{|N|}\Big)
\ll \frac{n}{a_n^d}\frac{\log|N|}{|N|}.
\end{align*}
This completes the proof.
\end{proof}

\begin{lemma} \label{lem:PiLLD}
There exists $C>0$ such that
\[
\Big\|\int_{B_\delta(0)} e^{-it\cdot N}r(t)\lambda_{k,t}^n\Pi_{k,t}\,dt \Big\|\le 
C\frac{n}{a_n^d} \frac{\log |N|}{|N|^2}
\]
for all $n\ge1$, $|N|> \pi/\delta$ with $a_n\le |N|\le e^{\eps_1 n}$,
$k=0,\dots,q-1$.
\end{lemma}

\begin{proof}
In this proof we abbreviate $B_\delta(0)$ to $B_\delta$ and suppress $dt$.
Let $I=\int_{B_\delta} e^{-it\cdot N}r(t)\lambda_{k,t}^n\Pi_{k,t}$. 
Integrating by parts, 
\[
I=\frac{1}{iN_j}
\int_{B_\delta} e^{-it\cdot N}\partial_jr(t)\lambda_{k,t}^n\Pi_{k,t}
+\frac{1}{iN_j}\int_{B_\delta} e^{-it\cdot N}r(t)\partial_j(\lambda^n\Pi)_{k,t}=I_1+I_2+I_3
\]
where
\begin{align*}
I_1 & =-\frac{1}{N_j^2}\int_{B_\delta} e^{-it\cdot N}\partial_j^2r(t)\lambda_{k,t}^n\Pi_{k,t}, \quad
I_2  = -\frac{1}{N_j^2}\int_{B_\delta} e^{-it\cdot N}\partial_jr(t)\partial_j(\lambda^n\Pi)_{k,t}, \\
I_3  &=
\frac{1}{iN_j}\int_{B_\delta} e^{-it\cdot N}r(t)\partial_j(\lambda^n\Pi)_{k,t}.
\end{align*}
Recall that $r$ is $C^2$ and that $t\mapsto\lambda_{k,t}$, $t\mapsto \Pi_{k,t}$ are $C^1$ by
Lemma~\ref{lem:key}.
Hence by Corollary~\ref{cor:int},
\[
\|I_1\|
\ll \frac{1}{|N|^2}\int_{B_\delta} |\lambda_{k,t}|^n
\ll \frac{1}{a_n^d}\frac{1}{|N|^2}, \qquad
\|I_2\|  \ll 
\frac{n}{|N|^2}\int_{B_\delta} |\lambda_{k,t}|^{n}
\ll \frac{n}{a_n^d}\frac{1}{|N|^2}.
\]

To estimate $I_3$, we use a modulus of continuity argument (see for instance~\cite[Chapter~1]{Katzn}).
Set $s=t-h$ where $h=\pi N_j^{-1}e_j$ and
notice that $I_3=
-\frac{1}{iN_j}\int_{B_{2\delta}} e^{-it\cdot N}r(s)\partial_j(\lambda^n\Pi)_{k,s}$.
Hence
\[
I_3  =
\frac{1}{2i|N_j|}\int_{B_{2\delta}} e^{-it\cdot N}(r(t)\partial_j(\lambda^n\Pi)_{k,t}-r(s)\partial_j(\lambda^n\Pi)_{k,s}).
\]
Setting
\(
I_4=\frac{1}{|N|}\int_{B_{2\delta}} |r(s)|\|\partial_j(\lambda^n\Pi)_{k,t}-\partial_j(\lambda^n\Pi)_{k,s}\|,
\)
we obtain
\begin{align*}
\|I_3\| & \ll 
\frac{1}{|N|}\int_{B_{2\delta}} |r(t)-r(s)|\|\partial_j(\lambda^n\Pi)_{k,t}\| +
I_4
\\ & \ll\frac{n}{|N|^2}\int_{B_{2\delta}} |\lambda_{k,t}|^{n} +I_4
\ll \frac{n}{a_n^d}\frac{1}{|N|^2}+I_4.
\end{align*}
Now,
\begin{align*}
I_4
& \ll \frac{1}{|N|}\int_{B_{2\delta}} \|\partial_j(\lambda^n)_{k,t}\Pi_{k,t}-\partial_j(\lambda^n)_{k,s}\Pi_{k,s}\|
+ \frac{1}{|N|}\int_{B_{2\delta}} \|\lambda_{k,t}^n\partial_j\Pi_{k,t}-\lambda_{k,s}^n\partial_j\Pi_{k,s}\|
\\
& \ll \frac{1}{|N|}\int_{B_{2\delta}} |\partial_j(\lambda^n)_{k,t}-\partial_j(\lambda^n)_{k,s}|\|\Pi_{k,t}\|
+
 \frac{1}{|N|}\int_{B_{2\delta}} |\partial_j(\lambda^n)_{k,s}|\|\Pi_{k,t}-\Pi_{k,s}\|
\\ & \qquad \qquad\qquad + \frac{1}{|N|}\int_{B_{2\delta}} |\lambda_{k,t}^n-\lambda_{k,s}^n|\|\partial_j\Pi_{k,t}\|
+ \frac{1}{|N|}\int_{B_{2\delta}} |\lambda_{k,s}|^n  \|\partial_j\Pi_{k,t}-\partial_j\Pi_{k,s}\|
\\
& \ll \frac{1}{|N|}\int_{B_{2\delta}} |\partial_j(\lambda^n)_{k,t}-\partial_j(\lambda^n)_{k,s}|
+
 \frac{n}{|N|^2}\int_{B_{2\delta}} |\lambda_{k,s}|^{n}
\\ & \qquad + \frac{n}{|N|^2}\int_{B_{2\delta}} |\lambda_{k,t}|^{n}
+ \frac{1}{|N|}\int_{B_{2\delta}} |\lambda_{k,s}|^n \|\partial_j\Pi_{k,t}-\partial_j\Pi_{k,s}\|  .
\end{align*}
To complete the proof, we show that
\(
I_4\ll
\frac{n}{a_n^d}\frac{\log|N|}{|N|^2}.
\)
The first integral on the right-hand side was estimated in Proposition~\ref{prop:worst} while the second and third
are dominated by 
$\frac{n}{a_n^d}\frac{1}{|N|^2}$.
The same calculation that was used for the integral $K$ in Proposition~\ref{prop:worst} shows that
$\int_{B_{2\delta}} |\lambda_{k,s}|^n \|\partial_j\Pi_{k,t}-\partial_j\Pi_{k,s}\|  
\ll \frac{1}{a_n^d}\frac{\log|N|}{|N|}$.
The desired estimate for $I_4$ follows.~
\end{proof}

\begin{cor} \label{cor:lld}
There exist $\eps_1>0$, $C>0$ such that
\[
\|A_{n,N}\|\le C\frac{n}{a_n^d} \frac{\log |N|}{1+|N|^2}
\quad\text{for all $n\ge1$, $N\in\Z^d$ with $|N|\le e^{\eps_1 n}$.}
\]
\end{cor}

\begin{proof}
By Corollary~\ref{cor:smallN},
$\|A_{n,N}\|\ll a_n^{-d}$.
Hence for $n\gg (1+|N|^2)/\log|N|$,
\[
\|A_{n,N}\|\ll a_n^{-d} \ll \frac{n}{a_n^d}\frac{\log|N|}{1+|N|^2}.
\]
Hence we can reduce to the case 
$n\le \frac12 (1+|N|^2)/\log|N|$.
Then we can suppose without loss that $|N|> \pi/\delta$. 
In this way, we reduce to proving 
$\|A_{n,N}\|\ll \frac{n}{a_n^d} \frac{\log |N|}{|N|^2}$
under the constraints
\[
|N|> \pi/\delta,\qquad |N|\le e^{\eps_1 n}, \qquad
n\le |N|^2/\log|N|.
\]
Since $a_n^2/\log a_n\sim n/2$, the last constraint can be weakened to
$a_n^2/\log a_n\le |N|^2/\log|N|$, equivalently $a_n\le |N|$.

By~\eqref{eq:Sz-bis},
there exists $\gamma\in(0,1)$ such that
\(
\|\int_{B_\delta(0)} e^{-it\cdot N}r(t)Q_t^n\,dt \|\ll 
\gamma^n .
\)
Together with Lemma~\ref{lem:PiLLD}, this implies that
\[
\|A_{n,N}\|\ll
\frac{n}{a_n^d} \frac{\log |N|}{|N|^2}
+\frac{1}{|N|^2} \gamma^n|N|^2.
\]
Shrinking $\eps_1$ if necessary,
$\gamma^n|N|^2\le \gamma^n e^{2\eps_1 n}\le
\gamma^{n/2} \ll \frac{n}{a_n^d}$,
and so $\|A_{n,N}\|\ll
\frac{n}{a_n^d} \frac{\log |N|}{|N|^2}$.~
\end{proof}

\begin{pfof}{Theorem~\ref{thm:lld}}
By Lemma~\ref{lem:bigN}, it remains to consider the range $\log|N|\le\eps_1 n$.
By~\cite[Lemma~3.9]{MTsub}, there exists an even $C^2$ function $r:\R^d\to\R$ supported in $B_\delta(0)$ such that
\[
1_{\{\kappa_n=N\}}\le \int_{\R^d}e^{-it\cdot N}r(t)e^{it\cdot \kappa_n}\,dt
\]
for $n\ge1$, $N\in\Z^d$.
Hence 
\[
P^n 1_{\{\kappa_n=N\}}\le \int_{\R^d}e^{-it\cdot N}r(t)P^n e^{it\cdot \kappa_n}\,dt=A_{n,N} \, 1_\Delta.
\]
 It follows that
\(
\mu_\Delta (\kappa_n=N)=\int_\Delta P^n 1_{\{\kappa_n=N\}}\,d\mu_\Delta
\ll \|A_{n,N}\|.
\)
By Corollary~\ref{cor:lld}, we obtain the desired estimate  for
$\log|N|\le\eps_1 n$.
\end{pfof}

\section{LLD for nonuniformly hyperbolic systems modelled by Young towers}
\label{sec:abs}

In this section, we state and prove an abstract version of Theorem~\ref{thm:lld}
for systems modelled by a Young tower with exponential towers for a general class of observables~$\kappa$. The observables take values in $\Z^d$ where there is no restriction on the value of $d\ge1$.

Let $(T,M,\mu)$ be a general nonuniformly hyperbolic map modelled by a two-sided Young tower $\Delta$ with exponential tails (as in Section~\ref{sec:setup}).
 Let $\kappa:M\to\Z^d$ be 
an integrable observable with $\int_M \kappa\,d\mu=0$.
Define the lifted observable
$\hat\kappa=\kappa\circ\pi:\Delta\to\Z^d$.
We require that $\hat\kappa$ is constant on $\bar\pi^{-1}(a\times\{\ell\})$ for each $a\in\alpha$, $\ell\in\{0,\dots,\sigma(a)$.  Then $\hat\kappa$ projects to an observable $\bar\kappa:\bar\Delta\to\Z^d$ constant on the partition elements $a\times\{\ell\}$ of the one-sided tower $\bar\Delta$.

Define $P_t$, $\lambda_{k,t}$ and so on as in Section~\ref{sec:setup}.
Properties~\eqref{eq:Sz} and~\eqref{eq:Sz-bis} remain valid.
Our further assumptions in the abstract setting are
that there exist continuous slowly varying\footnote{So $\lim_{t\to\infty}\ell_1(\lambda t)/\ell_1(t)=1$ for all $\lambda>0$, and similarly for $\ell_2$.} functions $\ell_1,\,\ell_2:[0,\infty)\to(0,\infty)$ and a constant $\delta>0$ such that
\begin{align} \label{eq:ass1}
& \mu(|\kappa|> x)  \le  x^{-2}\ell_1(x)
\quad\text{for all $x>1$} \\[.75ex]
& |\lambda_{0,t}-1|  \le  |t|^2\ell_2(1/|t|)
\quad\text{for all $t\in B_\delta(0)$.} \label{eq:ass2}
\end{align}
Define the slowly varying function
$\tilde\ell_1(x)=\int_1^{1+x} u^{-1}\ell_1(u/\log u)\,du$.\footnote{To optimise the results, we should take $\tilde\ell_1(x)=\int_1^{1+x} u^{-1}(\log u)^2\ell_1(u/\log u)\,du$. Then $\tL(t)=\tilde\ell_1(1/|t|)$ below but the formula for $\tM_b$ is much more complicated.}
We require that there is a constant $C>0$ such that 
\begin{equation} \label{eq:ass3}
(\log x)^2\tilde\ell_1(x)\le C\ell_2(x)
\quad\text{for all $x>1$}.
\end{equation}
Choose $a_n$ so that 
\[
na_n^{-2}\ell_2(a_n)\sim1.
\]

\begin{thm}[LLD in abstract setting]
\label{thm:abs} 
Let $d\ge1$.
There exist $C>0$ and a slowly varying function $\ell_3$ (depending on $\ell_1$, $\ell_2$ and $d$) such that
\[
\mu(\kappa_n=N)\le C
\frac{n}{a_n^d} \, \frac{\ell_3(|N|)}{1+|N|^2}
\quad\text{for all  $n\ge1$, $N\in\Z^d$.}
\]
\end{thm}

\begin{rmk} The slowly varying function $\ell_3$ can be determined by modifying the proof of Theorem~\ref{thm:lld}. Some of the steps are indicated below.

In the case of billiards,
assumptions~\eqref{eq:ass1} and~\eqref{eq:ass2} hold with $\ell_1\equiv1$ and $\ell_2(x)=\log x$.
We note that even with these $\ell_1$, $\ell_2$ and $d\le 2$, obtaining $\ell_3(x)=\log x$ and $\ell_3(x)=\log x \log\log x$ in Theorem~\ref{thm:lld} requires extra structure for billiards beyond the abstract setting of Theorem~\ref{thm:abs}.
This extra structure was used in Proposition~\ref{prop:psi} and Lemma~\ref{lem:bigN}.
Similarly, assumption~\eqref{eq:ass3} is not required in the billiard setting due to the extra structure.
\end{rmk}

\begin{rmk} (a) In the simpler situation of Gibbs-Markov maps studied in~\cite{MTsub}, the underlying assumption is that $\mu(|\kappa|>x)\sim x^{-2}\ell_1(x)$
and a consequence is that $1-\lambda_{0,t}\sim c|t|^2\ell_2(1/|t|)$ where $c>0$ and 
$\ell_2(x)=1+\int_1^{1+x}\ell_1(u)/u\,du$.
Moreover, $\ell_3=\ell_2$.
\\[.75ex]
\noindent (b)
As in~\cite{MTsub}, the proof of Theorem~\ref{thm:abs}
does not rely on aperiodicity assumptions and hence the result applies in situations where the local limit theorem fails.
(In fact, it may even be the case that $a_n^{-1}\kappa_n$ fails to converge in distribution in the generality of Theorem~\ref{thm:abs}.)
\\[.75ex]
\noindent (c) More generally, one could consider LLD in the abstract setting with
\begin{align*}
& \mu(|\kappa|> x)  \le C x^{-\alpha}\ell_1(x)
\quad\text{for all $x>1$}, \\[.75ex]
& |\lambda_{0,t}-1|  \le C |t|^\alpha\ell_2(1/|t|)
\quad\text{for all $t\in B_\delta(0)$},
\end{align*}
for $\alpha\in(0,2)$, where the underlying limit laws are stable laws (rather than normal distributions with nonstandard normalisation).
We already mentioned that the study of such stable LLD started with~\cite{CaravennaDoney} and~\cite{Berger19}
in the i.i.d.\ case for $d=1$, extended to $d\ge2$~\cite{Berger19b}.
The Gibbs-Markov case was studied in~\cite{MTsub} for $\alpha\in(0,1)\cup(1,2]$
and general $d\ge1$.   
We expect that Theorem~\ref{thm:abs},
in the abstract setting where $M$ is modelled by a Young tower with exponential tails,
extends to the cases $\alpha\in(0,1)\cup(1,2)$ with minor (and obvious) modifications.  However, for purposes of readability we do not pursue this extension here.
\end{rmk}

In the remainder of this section, we sketch the proof of Theorem~\ref{thm:abs}.
Again, the range
$n\ll \log|N|$ is handled at the level of $T:M\to M$ and $\kappa:M\to\Z^d$.
 Lemma~\ref{lem:bigN} is replaced by
\begin{lemma} \label{lem:bigNabs}
Let $d\ge1$, $\omega>0$, $\eps>0$.
There exists $C>0$ such that
\[
\mu(\kappa_n=N)\le C\frac{n}{a_n^d}\frac{\ell_1(|N|)(\log|N|)^{\frac{d}{2}+1+\eps}}{|N|^2}
\]
for all $n\ge1$, $N\in\Z^d$ with $n\le\omega\log|N|$.
\end{lemma}

\begin{proof}
Define $\tilde\kappa=\tilde\kappa(N)=\min\{|\kappa|,|N|\}$
and $M_n=\max_{0\le j\le n-1}|\kappa|\circ T^j$.
We use $|x|=\max_{j=1,\dots,d}|x_j|$ so that $|\tilde\kappa|$ is integer-valued.

Now,
\[
\mu(|\kappa_n|\ge |N|)\le 
\mu(|\kappa_n|\ge |N|,\,M_n\le |N|)+
\mu(M_n> |N|).
\]
Note that
\[
\mu(M_n> |N|)
\le
\sum_{j=0}^{n-1}\mu(|\kappa|\circ T^j> |N|)
=n\mu(|\kappa|> |N|)\ll n|N|^{-2}\ell_1(|N|).
\]
Next, for any $r>2$,
\[
\mu(|\kappa_n|\ge |N|,\,M_n\le |N|)\le 
\mu(\tilde\kappa_n\ge |N|)\le 
\|\tilde\kappa_n\|_r^r/|N|^r
\le n^r\|\tilde\kappa\|_r^r/|N|^r.
\]
By resummation and~\eqref{eq:ass1},
\[
\|\tilde\kappa\|_r^r=\sum_{j=1}^\infty j^r\mu(|\tilde\kappa|=j)
\le \sum_{j=1}^{|N|} j^r\mu(|\kappa|=j)
\ll \sum_{j=1}^{|N|} j^{r-1}\mu(|\kappa|>j)
\ll \sum_{j=1}^{|N|} j^{r-3}\ell_1(j),
\]
so by Karamata, $\|\tilde\kappa\|_r^r\ll |N|^{r-2}\ell_1(|N|)$.
Hence
\[
\mu(\kappa_n=N)\le \mu(|\kappa_n|\ge |N|)\le n^r|N|^{-2}\ell_1(|N|)
= n a_n^{-d}|N|^{-2}\ell_1(|N|)n^{r-1}a_n^d.
\]
Since $a_n$ is regularly varying of index $\frac12$, and $r>2$ is arbitrary, it follows that 
$n^{r-1}a_n^d\ll 
n^{\frac{d}{2}+1+\eps}
\ll (\log|N|)^{\frac{d}{2}+1+\eps}$.
\end{proof}

The remainder of the proof of Theorem~\ref{thm:abs} is carried out on the one-sided tower. 
As in Section~\ref{sec:setup}, we define $\psi(y)=\sum_{\ell=0}^{\sigma(y)-1}|\kappa(y,\ell)|$. The analogue of Proposition~\ref{prop:psi} is:

\begin{prop}  \label{prop:psiabs}
There exist $C,\,n_0>1$ such that
\[
\mu_Y(\psi>n)\le C n^{-2}(\log n)^2\ell_1( n/\log n)
\quad\text{for all $n\ge n_0$.}
\]
In particular,
$\psi\in L^r(Y)$ for all $r<2$.
\end{prop}

\begin{proof}
A standard argument (see for example~\cite[Proposition~A.1]{BM18})
shows that 
\[
\mu_Y(\psi>n)\le \mu_Y(\sigma>k)+\bar\sigma\mu(|\kappa|>n/k),
\]
for $k,n>1$. In particular, there exists $a>0$ such that
\[
\mu_Y(\psi>n)\ll e^{-ak}+n^{-2}k^2\ell_1(n/k).
\]
Taking $k=q\log n$ for any $q>2/a$ and using that $\ell_1$ is slowly varying,
\[
\mu_Y(\psi>n)\ll n^{-2}(\log n)^2\ell_1( n/\log n).
\]

Let $\eps\in(0,2-r)$. Since $\ell_1$ is slowly varying, $\ell_1(n/\log n)\ll 
(n/\log n)^{\eps/2}\ll n^{\eps/2}$. 
Hence $\mu_Y(\psi>n)\ll n^{-(2-\eps)}$ and it follows that $\psi\in L^r$.
\end{proof}

Define 
\[
\tM_b(t,h)=|h|\tL(h)\big\{1+
L(h)|t|^2\tL(t)
+|h|^{-b|t|^2\tL(t)}L(h)^2|t|^4\tL(t)^2
\big\},
\]
where $L(t)=\log(1/|t|)$ and $\tL(t)=L(t)^2\tilde\ell_1(1/|t|)$.

\begin{lemma} \label{lem:keyabs}
The conclusions of Lemma~\ref{lem:key} and~\ref{lem:key2} hold with
$M_b$ and $L(t)$ replaced by $\tM_b$ and $\ell_2(1/|t|)$ respectively.
\end{lemma}

\begin{proof}
The modifications are elementary, but heavy on notation, so we only sketch the details.

Since $\psi\in L^r$ for all $r<2$, the arguments in Section~\ref{sec:Y} are unchanged.
The changes in the proof of Proposition~\ref{prop:furtherR} are as follows.
By resummation, Proposition~\ref{prop:psiabs} and the definition of $\tilde\ell_1$,
\begin{align*}
\sum_{m=1}^K\mu_Y(\psi=m)m^2 & \ll 
\sum_{m=1}^K\mu_Y(\psi\ge m)m\ll 
\sum_{m=1}^K m^{-1}(\log m)^2\ell_1(m/\log m)
\\ & \ll (\log K)^2\sum_{m=1}^K m^{-1}\ell_1(m/\log m)
\ll (\log K)^2\tilde\ell_1(K).
\end{align*}
Using this in~\eqref{eq:est1}, we obtain 
\[
\sum_{m=1}^K\sum_{n\le b\log m}r_{m,n} \ll
|h| L(h)^3\tilde\ell_1(1/|h|)\big\{1+ (|z|-1)|h|^{-b\log|z|}L(h)\big\}.
\]
Similarly,
\[
\sum_{m>K}\sum_{n\le b\log m} r_{m,n}
\ll |h|L(h)^3\ell_1(|h|^{-1}L(h)^{-1})\big\{1+ (|z|-1)|h|^{-b\log|z|}L(h)\big\}.
\]
Hence the estimate corresponding to Proposition~\ref{prop:furtherR} is 
\[
\|\partial_j\partial_z \hR(z,t+h)
-\partial_j\partial_z \hR(z,t)\|_{\cB_1(Y)}
  \ll  |h|L(h)^3\tilde\ell_1(1/|h|)\big\{1+
(|z|-1) |h|^{-b\log |z|}L(h)\big\}.
\]
The corresponding estimates for $\partial_j\hR$, $\partial_j\hA$ and $\partial_j\hB$ are the same but with one less factor of $L(h)$.
%Also,
%\[
%\|\partial_j\hR(1,t+h) -\partial_j\hR(1,t)\|_{\cB_1(Y)} \ll  |h| L(h)^2\tilde\ell_1(1/h).
%\]

Parts~(a) and~(c) of Proposition~\ref{prop:mu0} are unchanged. Part~(b)
goes through with
$L$ replaced by $\tL$.
Hence, Corollary~\ref{cor:g} becomes that
\[
|g_k(t)-\bar\lambda_k|  \ll |t|^2 \tL(t), \qquad
|\partial_j g_k(t+h)-\partial_j g_k(t)|  \ll \tM_b(t,h).
\]
The result follows.
\end{proof}

\begin{cor} \label{cor:intabs}
Let $\beta\ge0$, $r\in\R$, $k=0,\dots,q-1$.
There exist $C>0$, $\delta>0$ such that
\[
\int_{B_{2\delta}(0)} |t|^\beta \tL(t)^r|\lambda_{k,t}|^n \,dt\le 
C\frac{(\tL(1/a_n))^r}{a_n^{d+\beta}}
\quad\text{for all $n\ge1$}. 
\]
\end{cor}

\begin{proof} Following the proof of Corollary~\ref{cor:int}, we obtain $|\lambda_{k,t}|\le \exp\{-b|t|^2\ell_2(t)\}$.  Now use that $a_n$ is defined using $\ell_2$ instead of $L$.
\end{proof}

\begin{pfof}{Theorem~\ref{thm:abs}}
The arguments are identical to those in Section~\ref{sec:lldpf} up to slowly varying factors.  Various simplifications no longer hold as the slowly varying functions $\ell_1$, $\ell_2$, $\tilde\ell_1$ and $\log$ are less well related, so the exact formulas are rather complicated and hence are omitted.
\end{pfof}

\paragraph{Acknowledgements}
DT was partially supported by EPSRC grant EP/S019286/1.

\end{document}